\documentclass[a4paper,11pt]{article}

\usepackage{lmodern}
\usepackage[T1]{fontenc}
\usepackage{hyperref}
\usepackage{amsfonts,amsmath,amssymb,amsopn,amsthm}
\usepackage[all]{xy}
\usepackage{vmargin}
\usepackage{makeidx}
\usepackage{graphicx}
\usepackage{tikz}
\usepackage{comment}
\usepackage{mathabx}
\title{Twisted divided powers and applications}
\author{Michel Gros, Bernard Le Stum \& Adolfo Quir\'os\thanks{Supported by grant MTM2015-68524-P (MINECO/FEDER, UE).}}
\date{Version of \today}
\usepackage[ style=alphabetic,citestyle=alphabetic,sorting=nyt,sortcites=true,autopunct=true,babel=hyphen,hyperref=true,abbreviate=false,backref=true]{biblatex}
\AtEveryBibitem{ \clearfield{url} \clearfield{urldate} \clearfield{review} \clearfield{series} \clearfield{doi} \clearfield{isbn} \clearfield{issn} } 
\defbibheading{bibempty}{}
\DeclareNolabel{\nolabel{\regexp{[\p{Z}]+}}}
\newtheorem{thm}{Theorem}[section]
\newtheorem{prop}[thm] {Proposition}
\newtheorem{cor}[thm] {Corollary}
\newtheorem{lem}[thm] {Lemma}
\newtheorem{dfn}[thm] {Definition}

\newenvironment{xmp}[1][Example]{\begin{trivlist} \item[\hskip \labelsep {\bfseries #1}]}{\end{trivlist}}

\newenvironment{rmk}[1][Remark]{\begin{trivlist} \item[\hskip \labelsep {\bfseries #1}]}{\end{trivlist}}
\newenvironment{rmks}[1][Remarks]{\begin{trivlist} \item[\hskip \labelsep {\bfseries #1}]}{\end{trivlist}}
\begin{document}

\maketitle

\begin{abstract}
In order to give a formal treatment of differential equations in positive characteristic $p$, it is necessary to use divided powers.
One runs into an analog problem in the theory of $q$-difference equations when $q$ is a $p$th root of unity.
We introduce here a notion of twisted divided powers (relative to $q$) and show that one can recover the twisted Weyl algebra and obtain a twisted $p$-curvature map that describes the center of the twisted Weyl algebra.
We also build a divided $p$-Frobenius that will give, by duality, a formal Azumaya splitting of the twisted Weyl algebra as well as a twisted Simpson correspondence.
\end{abstract}

\tableofcontents

\section*{Introduction}
\addcontentsline{toc}{section}{Introduction}

\subsection*{Motivation}
The classical Simpson correspondence establishes an equivalence between certain local systems and certain Higgs bundles (see \cite{Simpson92}).
It is purely topological in nature.
There exists also a Simpson correspondence in positive characteristic (see \cite{OgusVologodsky07}) that we recall now (local form):

\begin{thm}[Ogus-Vologodsky] \label{OV}
Let $S$ be scheme of positive characteristic $p$ and $X$ a smooth scheme over $S$.
Then, if the relative Frobenius $F : X \to X'$ lifts modulo $p^2$, it induces an equivalence between modules with a quasi-nilpotent integrable connection on $X$ and quasi-nilpotent Higgs bundles on $X'$.
\end{thm}

In \cite{GrosLeStumQuiros10}, we generalized this theorem to higher level with a strategy of proof that was different from the original one.
We want to recall it here.
Let us denote by $\mathcal D^{(0)}_{X}$ the ring of differential operators of level zero (Berthelot's sheaf of differential operators) of $X/S$ and by $\mathcal T_{X'}$ the tangent sheaf on $X'/S$.
Then, an $\mathcal O_X$-module with a quasi-nilpotent integrable connection is the same thing as a $\widehat{\mathcal D}^{(0)}_{X}$-module, and a quasi-nilpotent Higgs bundle on $X'$ is the same thing a $\widehat{S^\bullet\mathcal T}_{X'}$-module (where $S^\bullet$ denotes the symmetric algebra and completion is always meant with respect to the augmentation ideal).
Moreover, there exists an injective \emph{$p$-curvature map} $S^\bullet\mathcal T_{X'} \hookrightarrow \mathcal D^{(0)}_{X}$ whose image is exactly the center $\mathcal Z^{(0)}_X$ of $\mathcal D^{(0)}_{X}$ ; and the image of the \emph{linearized $p$-curvature map} $\mathcal O_{X} \otimes_{\mathcal O_{X'}}S^\bullet\mathcal T_{X'} \hookrightarrow \mathcal D^{(0)}_{X}$ is the centralizer $\mathcal Z\mathcal O_{X}^{(0)}$ of $\mathcal O_X$.
Using a lifting of Frobenius, one can build an isomorphism
\[
\widehat{\mathcal D}^{(0)}_{X} \simeq \mathcal End_{\widehat{\mathcal Z}^{(0)}_{X}}(\widehat{\mathcal Z\mathcal O_{X}}^{(0)})
\]
from which Simpson correspondence may be deduced through Morita equivalence.
Actually, if $\mathcal P^{(0)}_{X}$ denotes the ring of principal parts of level zero of $X/S$ and $\Omega^1_{X'}$ is the sheaf of differential forms on $X'/S$, then this isomorphism comes by duality from an isomorphism
\begin{equation} \label{upone}
\mathcal O_{X \times_{X'} X} \otimes_{\mathcal O_{X'}} \Gamma_\bullet \Omega^1_{X'} \simeq \mathcal P_X^{(0)}
\end{equation}
(where $\Gamma_\bullet$ denotes the divided power algebra).

The key of the construction consists in using a lifting $\tilde F$ of $F$ modulo $p^2$ in order to define the \emph{divided Frobenius map},
\[
[F^*]  := \frac 1{p}\widetilde F : \Omega^1_{X'} \to \mathcal P_X^{(0)},
\]
that can be extended in order to obtain the isomorphism \eqref{upone}.
Let us also recall how the $p$-curvature map may be obtained by duality.
If $\mathcal I_{X}$ (resp. $\mathcal I_{X'}$) denotes the ideal of the diagonal of $X/S$ (resp. $X'/S$), then one can use the divided power map
\[
\varphi \mapsto \varphi^{[p]}, \quad \mathcal I_{X'} \to \mathcal P_X^{(0)}
\]
in order to define a morphism $\Omega^1_{X'} = \mathcal I_{X'}/\mathcal I_{X'}^2 \to \mathcal P_X^{(0)}/\mathcal I_X\mathcal P_X^{(0)}$.
In fact, we obtain an isomorphism
\begin{equation}\label{downone}
\mathcal O_{X} \otimes_{\mathcal O_{X'}}\Gamma_\bullet\Omega^1_{X'} \simeq \mathcal P_X^{(0)}/\mathcal I_X\mathcal P_X^{(0)}
\end{equation}
and the linearized $p$-curvature is dual to the following composition
\[
\mathcal P_X^{(0)} \twoheadrightarrow \mathcal P_X^{(0)}/\mathcal I_X\mathcal P_X^{(0)} \simeq \mathcal O_{X} \otimes_{\mathcal O_{X'}}\Gamma_\bullet\Omega^1_{X'}.
\]

Let us give an explicit description of these constructions.
Locally, we may assume that $S = \mathrm{Spec}(R)$ and $X = \mathrm{Spec}(A)$ are affine and that we are given a system of \'etale coordinates on $X$.
Actually, we will concentrate on the one dimensional case and call $x$ the coordinate.
The pull back $A'$ of $A$ along the Frobenius of $R$ comes with an \'etale coordinate $x'$.
We denote by $F^* : A' \to A$ the relative Frobenius of $A$ so that $F^*(x')= x^p$.
We let $\xi = 1 \otimes x - x \otimes 1 \in A \otimes_{R} A$ and denote by $\omega \in \Omega^1_{A'}$ the class of $\xi' = 1 \otimes x' - x' \otimes 1 \in A' \otimes_{R} A'$.
If we write $A\langle \xi \rangle$ and $A\langle \omega \rangle$ for the divided power polynomial rings, then the isomorphism \eqref{downone} is the $A$-linear map
\begin{equation} \label{dupcu}
A\langle \omega \rangle \simeq A\langle \xi \rangle/\xi, \quad \omega^{[k]} \mapsto \overline{\xi^{[pk]}}.
\end{equation}
We can also describe the divided Frobenius map when we are given a lifting $\widetilde F^*$ of $F^*$ modulo $p^2$.
To make it simpler, we assume that $\widetilde F^*(\widetilde x') = \widetilde x^p$.
Then, from $\widetilde F^*(\widetilde \xi') = 1 \otimes \widetilde x^p - \widetilde x^p \otimes 1$ one easily derive
\begin{equation} \label{formF}
[F^*](\omega) = \sum_{i=1}^{p} (p-1) \cdots (p-i+1) x^{p-i}\xi^{[i]}
\end{equation}
and the isomorphism \eqref{upone} is given by
\[
(A[\xi]/\xi^p) \langle \omega \rangle \simeq A\langle \xi \rangle, \quad \omega^{[k]} \mapsto ([F^*](\omega))^{[k]}.
\]

We will mimic this strategy in the twisted case and prove in the end the following theorem (the vocabulary will be specified later on):
 
\begin{thm} \label{GLSQ}
Let $R$ be a commutative ring and $q \in R$ such that $R$ is $q$-divisible of $q$-characteristic $p > 0$.
Let $(A, \sigma)$ be a twisted $R$-algebra with twisted coordinate $x$ such that $\sigma(x) = qx$.
If $F^*$ is a $p$-Frobenius on $A$ which is \emph{adapted} to $\sigma$, then it induces an equivalence between $A$-modules endowed with a quasi-nilpotent $\sigma$-derivation and $A'$-modules endowed with a quasi-nilpotent Higgs field.
\end{thm}

Let us make some comments. The condition that $R$ is $q$-divisible of $q$-characteristic $p > 0$ is satisfied for example in the following situations:
\begin{enumerate}
\item $q = 1$ and $\mathrm{Char}(R) = p$ with $p$ prime: this is Ogus-Vologodsky's theorem,
\item $q \neq 1$ and $q$ is a $p$th root of unity with $p$ prime,
\item $q \in K \subset R$, with $K$ a field, is a primitive $p$th root of unity but $p$ needs not be prime.
\end{enumerate}
Then if we are given an $R$-algebra $A$, the existence (and uniqueness) of $\sigma$ and $F^*$ satisfying the above properties, are guaranteed in the following situations:
\begin{enumerate}
\item $A = R[x]$ or $A = R[x, x^{-1}]$ and $q \in R^\times$,
\item $R$ is $p^N$-torsion with $p$ prime (and the $p$th power map of $R/p$ lifts to $R$) and $x$ is an \'etale coordinate on $A$.
\end{enumerate}
In particular, we see that when $R$ is $p^N$-torsion with $p$ prime, theorem \ref{GLSQ} is a $q$-deformation of theorem \ref{OV} is the sense of \cite{Scholze16*}.

\subsection*{Description}

In the first section, we study the behavior under multiplication of twisted powers in a polynomial ring.
Roughly speaking, these twisted powers are the products that naturally appear when one writes down a formal solution for a $q$-differential equation.
They depend on the constant $q$ but also on the variable $x$.
Actually, for more flexibility, we use another parameter $y$ (which is $y := (1-q)x$ in practice).
The point is to check that there is enough divisibility in the sense of $q$-integers so that we can define the twisted divided power polynomial ring in section two.
We need these divided powers because we are mainly interested in the case when $q$ is a primitive root of unity where (twisted) powers are not sufficient.

Beware that there is no such thing as a general theory of twisted divided powers and we are only able to do the twisted divided power polynomials.
Nevertheless, we can define the twisted divided $p$-power map by using different parameters $q$ and $y$ on both sides, and give an explicit description of the image.
We will also show that, as in the classical case, there exists a duality between polynomials and twisted divided power polynomials.
In the third section, we apply the previous constructions to the case where there exists an endomorphism $\sigma$ that multiplies $y$ by $q$.
In this situation, there exists a general theory of twisted powers and it is compatible with the previous one.
We show that $\sigma$ extends to twisted divided power polynomials and that it behaves nicely with respect to twisted divided $p$-power map as well as duality.

In the fourth section, we introduce the twisted principal parts of level zero.
This is the ring where the formal solutions of a $q$-differential equation live, even when $q$ is a root of unity.
At this point, we really need a coordinate $x$ and set $y = x - \sigma(x)$.
Note that there exists a theory of twisted principal parts of infinite level that is sufficient when $q$ is not a root of unity.
However, we need twisted divided powers in order to obtain the correct object in general, exactly as what happens in positive characteristic for usual differential equations.
One can define formally the Taylor map and check that it is given by the expected formula.
Using this Taylor map, one can dualize the construction and define in section five the notion of twisted differential operator of level zero.
We show that, as expected, the ring of twisted differential operators of level zero is isomorphic to the twisted Weyl algebra.
In section six, we concentrate on the primitive $p$th root of unity situation.
One can then define the twisted $p$-curvature map as the dual of the twisted divided $p$-power map introduced earlier.
We show that its image is exactly the center of the twisted Weyl algebra.

Section seven is quite technical.
We want to define the notion of divided $p$-Frobenius on the twisted divided power polynomial rings (again, we need different flavors of the divided powers on the source and the target).
Actually, we were unable to give an explicit formula and will rely on a generic argument in order to show the existence of the map.
In the last section, we concentrate again on the root of unity situation and we dualize the twisted divided $p$-Frobenius map in order to obtain a formal Azumaya splitting of the twisted Weyl algebra.
It is then completely standard to derive by Morita equivalence a Simpson correspondence for twisted differential modules.

\subsection*{Notations}

Throughout the article, $R$ will denote a commutative ring (with unit) and $q$ will be a fixed element of $R$.
We need to recall here some vocabulary and notation from \cite{LeStumQuiros15}.
First of all, the \emph{$q$-analog} of a natural integer $m$ is:
\[
(m)_{q} := 1 + q + \cdots + q^{m-1}.
\]
And when $q \in R^\times$, the \emph{$q$-analog} of $-m$ is:
\[
(-m)_{q} := - \frac 1 q + \cdots +\frac 1 {q^{m}}.
\]
We will also call $(m)_{q}$ (or $(-m)_{q}$ when $q \in R^\times$) a \emph{$q$-integer} of $R$.

We might use the attribute ``twisted'' in place of the prefix $q$ and say twisted analog or twisted integer for example instead of $q$-analog or $q$-integer.
The same remark applies to all the forthcoming definitions.

The $q$-\emph{characteristic} of $R$ is the smallest positive integer $p$ such that $(p)_{q} = 0$ if it exists, and zero otherwise.
We will then write $q\mathrm{-char} (R) := p$.
If $q \neq 1$ and $p > 0$, then it means that $q$ is a primitive $p$th root of unity.
When $q = 1$, then $p$ is nothing than the usual characteristic of $R$.

The ring $R$ is said to be \emph{$q$-flat} (resp. \emph{$q$-divisible}) if $(m)_{q}$ is always regular (resp. invertible) in $R$ unless $(m)_{q} = 0$.
For example, when the $q$-characteristic $p$ is a prime number, then $R$ is automatically $q$-divisible, and therefore also $q$-flat.
And of course, when $R$ is a domain (resp. a field), then $R$ is automatically $q$-flat (resp. $q$-divisible).
More generally, it is sufficient to assume that $q$ belongs to a subdomain (resp. subfield) of $R$.

We also define the \emph{$q$-factorial} of $m \in \mathbb N$ as
\[
(m)_{q}! := (1)_{q} (2)_{q} \cdots (m)_{q}
\]
and, by induction, the \emph{$q$-binomial coefficients}
\[
{n \choose k}_{q} := {n-1 \choose k-1}_{q} + q^k{n-1 \choose k}_{q}
\]
when $n,k \in \mathbb N$.
Note that we recover the twisted analog as a special occurrence of a twisted binomial coefficient since
\[
(m)_{q} = {m \choose 1}_{q}
\]
if $m \in \mathbb N$.

\section{Twisted powers}

Recall that $R$ denotes a commutative ring and $q \in R$.
We assume in this section that $A$ is a commutative $R$-algebra (with unit) and we also fix some $y \in A$.

We denote by $A[\xi]$ the polynomial ring over $A$ and by $A[\xi]_{\leq n}$ the $A$-module of polynomials of degree at most $n$.
We set for all $n \in \mathbb N$,
\begin{equation}\label{naivpow}
\quad \xi^{(n)} := \prod_{i=0}^{n-1}\left(\xi + (i)_{q}y\right) \in A[\xi]_{\leq n}.
\end{equation}
If we want to make clear that these elements depend on $q$ and $y$, we might write $\xi^{(n)_{q,y}}$ but we will try to avoid as much as possible this clumsy notation.
As we will see later, notation \eqref{naivpow} is related to the twisted powers of \cite{LeStumQuiros15} but we do not need to know this at the moment.

Note that, by definition, we have
\[
\xi^{(0)} = 1, \quad \xi^{(1)} = \xi, \quad \ldots \quad \xi^{(n)} = \xi(\xi + y) \cdots (\xi + (n-1)_{q}y), \quad \ldots
\]
We will also use the induction formula
\begin{equation} \label{indform}
\xi^{(n+1)} = \xi^{(n)} (\xi + (n)_{q}y).
\end{equation}

\begin{lem} \label{bas}
The $\xi^{(n)}$'s for $n \in \mathbb N$ form a basis of the $A$-module $A[\xi]$.
More precisely, the $\xi^{(m)}$'s for $m \leq n$ form a basis of $A[\xi]_{\leq n}$.
 \end{lem}

In other words, the map $\xi^n \mapsto \xi^{(n)}$ defines an automorphism of $A[\xi]$ as filtered $A$-module (by the degree).

\begin{proof}
This follows from the fact that each $\xi^{(n)}$ is monic of degree $n$.
\end{proof}

\begin{lem}\label{funmul}
In $A[\xi]$, we have for all $m, n \in \mathbb N$,
\[
\xi^{(m)} \xi^{(n)} = \sum_{i=0}^{\min{(m,n)}} (-1)^i (i)_{q}! q^{\frac{i(i-1)}2}{m \choose i}_{q}{n \choose i}_{q} y^{i}\xi^{(m+n-i)}.
\]
\end{lem}

\begin{proof}
This is proved by induction on $n$.
The formula is trivially true for $n = 0$ and we will have
\begin{eqnarray}
\xi^{(m)} \xi^{(n+1)} & = & \xi^{(m)} \xi^{(n)} (\xi + (n)_{q}y)
\\
&= & \sum_{i=0}^{\min{(m,n)}} (-1)^i(i)_{q}! q^{\frac{i(i-1)}2}{m \choose i}_{q}{n \choose i}_{q} y^i\xi^{(m+n-i)}(\xi + (n)_{q}y). \label{bigone}
\end{eqnarray}

Now, we know from proposition 1.3 of \cite{LeStumQuiros15} that for all $0 \leq i \leq m+n$, we have
\[
(n)_{q} = (m+n-i)_{q} - q^n(m-i)_{q}.
\]
Therefore, we see that
\begin{eqnarray*}
\xi^{(m+n-i)}(\xi + (n)_{q}y) & = & \xi^{(m+n-i)}\left(\xi + (n+m-i)_{q}y - q^n(m-i)_{q}y\right)
\\
&= & \xi^{(m+n+1-i)} - q^n(m-i)_{q}y\xi^{(m+n-i)}.
\end{eqnarray*}
We can replace in \eqref{bigone} and get
\[
\xi^{(m)} \xi^{(n+1)} = S + T
\]
with
\[
S = \sum_{i=0}^{\min{(m,n)}} (-1)^i (i)_{q}! q^{\frac{i(i-1)}2}{m \choose i}_{q}{n \choose i}_{q} y^i\xi^{(m+n+1-i)} 
\]
and
\[
T =  - \sum_{i=0}^{\min{(m,n)}} (-1)^i (i)_{q}!q^{\frac{i(i-1)}2} {m \choose i}_{q}{n \choose i}_{q} y^i q^n(m-i)_{q}y\xi^{(m+n-i)}.
\]
Changing $i$ to $i-1$, we obtain
\[
T = \sum_{i=1}^{\min{(m,n)}+1} (-1)^{i}q^{n+1-i}(i-1)_{q}! (m+1-i)_{q}q^{\frac{i(i-1)}2} {m \choose i-1}_{q}{n \choose i-1}_{q} y^i\xi^{(m+n+1-i)}.
\]
Now we can compute for $1 \leq i \leq \min{(m,n)}$,
\[
 (i)_{q}! {m \choose i}_{q}{n \choose i}_{q} + q^{n-i+1}(i-1)_{q}! (m-i+1)_{q} {m \choose i-1}_{q}{n \choose i-1}_{q} 
\]
\[
= (i)_{q}! {m \choose i}_{q}{n+1 \choose i}_{q}.
\]
And the assertion will follow once we have checked the the side cases.
For $i = 0$, this should be clear and the case $i = \min(m,n)+1$ has to be split in two.
First, if $m \leq n$, then $i = m+1$ and $(m-i+1)_{q} = 0$: there is no contribution as expected.
Second, if $m > n$ and $i = n+1$, we do have
\[
(n)_{q}! (m-n)_{q} {m \choose n}_{q}{n \choose n}_{q} = (n+1)_{q}! {m \choose n+1}_{q}{n+1 \choose n+1}_{q}.\qedhere
\]\end{proof}
 
 \begin{rmks}
 \begin{enumerate}
 \item In the case $m=1$, we find
 \[
\xi \xi^{(n)} = \xi^{(n+1)} - (n)_{q} y\xi^{(n)}
\]
which we can also directly derive from the induction formula \eqref{indform}.
 \item The coefficients of $y^i\xi^{(m+n+i)}$ are polynomials in $q$ with integer coefficients.
 Actually, in order to prove the lemma, it would be sufficient to consider the case $R = \mathbb Z[t]$ and $q = t$.
 Or even $R = \mathbb Q(t)$.
 However, this does not seem to make anything simpler at this point.
 \item
In the case $q = 1$, we will rather write $\omega$ instead of $\xi$ for the extra variable.
Then, the multiplication formula simplifies a little bit to
\[
\omega^{(m)} \omega^{(n)} = \sum_{i=0}^{\min{(m,n)}} (-1)^i i! {m \choose i}{n \choose i} y^{i}\omega^{(m+n-i)}.
\]
\end{enumerate}
\end{rmks}

\begin{lem} \label{binai}
Assume that $q=1$.
Then, under the morphism of $A$-algebras
\begin{equation} \label{delna}
\xymatrix@R0cm{  A[\omega] \ar[r]^-\delta & A[\omega]\otimes_{A} A[\omega]
\\  \omega \ar@{|->}[r] & 1 \otimes \omega + \omega \otimes 1,}
\end{equation}
we have
\[
\delta\left(\omega^{(n)}\right) := \sum_{i=0}^{n} {n \choose i} \omega^{(n-i)} \otimes \omega^{(i)}.
\]
\end{lem}

\begin{proof}
The formula is proved to be correct by induction on $n$.
First of all, since $\delta$ is a ring homomorphism, we have
\[
\delta(\omega^{(n+1)}) = \delta\left(\omega^{(n)}(\omega + ny)\right) = \delta(\omega^{(n)})\delta(\omega + ny).
\]
Moreover, we can write for all $i = 0, \ldots, n$,
\[
\delta(\omega + ny) = 1 \otimes \omega + \omega \otimes 1 + ny = 1 \otimes (\omega + iy) + (\omega + (n-i)y) \otimes 1.
\]
Thus, by induction, we will have
\begin{eqnarray*}
\delta(\omega^{(n+1)})  &=& \sum_{i=0}^{n} {n \choose i} (\omega^{(n-i)} \otimes \omega^{(i)})(1 \otimes (\omega + iy) + (\omega + (n-i)y) \otimes 1)
\\
&=& \sum_{i=0}^{n} {n \choose i} \omega^{(n-i)} \otimes \omega^{(i)}(\omega + iy) + \sum_{i=0}^{n} {n \choose i}\omega^{(n-i)}(\omega + (n-i)y)  \otimes \omega^{(i)}
\\
&=& \sum_{i=0}^{n} {n \choose i} \omega^{(n-i)} \otimes \omega^{(i+1)} + \sum_{i=0}^{n} {n \choose i} \omega^{(n-i+1)}  \otimes \omega^{(i)}
\\
&=& \sum_{i=1}^{n+1} {n \choose i-1} \omega^{(n-i+1)} \otimes \omega^{(i)} + \sum_{i=0}^{n} {n \choose i} \omega^{(n-i+1)}  \otimes \omega^{(i)}
\\
&=& \sum_{i=0}^{n+1} \left({n \choose i-1} +  {n \choose i}\right) \omega^{(n+1-i)} \otimes \omega^{(i)}
\\
&=& \sum_{i=0}^{n+1} {n +1 \choose i} \omega^{(n+1-i)} \otimes \omega^{(i)}.\qedhere
\end{eqnarray*}\end{proof}

\section{Twisted divided powers} \label{first}

We let as before $A$ be a commutative $R$-algebra with a distinguished element $y$.

We denote by $A\langle\xi\rangle$ the free $A$-module on the (abstract) generators $\xi^{[n]}$ with $n \in \mathbb N$.
We will set $1 := \xi^{[0]}$ and $\xi := \xi^{[1]}$.
We will also denote by $I^{[n+1]}$ the free $A$-submodule generated by all $\xi^{[k]}$ with $k > n$ and
\[
A\langle\langle\xi\rangle\rangle := \varprojlim A\langle\xi\rangle/I^{[n+1]}.
\]
We will soon turn $A\langle\xi\rangle$ into a commutative $A$-algebra that will depend on $q$ and $y$.
If necessary, we will then write
\[
A\langle\xi\rangle_{q,y}, \quad \xi^{[n]_{q,y}}, \quad I_{q,y}^{[n+1]} \quad \mathrm{and} \quad A\langle\langle\xi\rangle\rangle_{q,y}.
\]

The next result is elementary but fundamental.

\begin{prop} \label{bijfond}
There exists a unique morphism of filtered $A$-modules
\begin{equation}\label{dppol}
\xymatrix@R0cm{ A[\xi] \ar[r] & A\langle\xi\rangle
\\ \xi^{(n)} \ar@{|->}[r] & (n)_{q }! \xi^{[n]}.}
\end{equation}
It is an isomorphism if all positive $q$-integers are invertible in $R$.
\end{prop}

The last condition means that $R$ is $q$-divisible of $q$-characteristic zero.

\begin{proof}
This follows from the facts that the $\xi^{(n)}$'s form a basis of $A[\xi]$ thanks to lemma \ref{bas}, and that the $\xi^{[n]}$'s form a basis of $A \langle\xi \rangle$ by definition.
\end{proof}

In the latest case, we will turn the bijection into an identification. In other words, we will write
\[
 \xi^{[n]} = \frac {\xi^{(n)}} {(n)_{q}!} = \frac {\xi(\xi + y) \cdots (\xi + (n-1)_{q}y)} {1 \cdots (n-1)_{q} (n)_{q}}.
\]

\begin{prop} \label{dpring}
The multiplication rule
\begin{equation} \label{sqform}
\forall m,n \in \mathbb N, \quad \xi^{[m]} \xi^{[n]} = \sum_{i=0}^{\min{(m,n)}} (-1)^iq^{\frac{i(i-1)}2}{m + n -i \choose m}_{q}{m \choose i}_{q} y^i\xi^{[m+n-i]}
\end{equation}
defines a structure of commutative $A$-algebra on $A\langle\xi\rangle$ and the linear map \eqref{dppol} is a morphism of $A$-algebras.
Moreover, for all $n \in \mathbb N$, $I^{[n+1]}$ is an ideal in $A\langle\xi\rangle$.
\end{prop}

Note that we have
\[
{m + n -i \choose m}_{q}{m \choose i}_{q} = {m + n -i \choose n}_{q}{n \choose i}_{q} 
\]
so that the formula is actually symmetric in $m$ and $n$.

\begin{proof}
In order to show that these formulas define a ring structure, it is sufficient to consider the case where $R = \mathbb Z[t]$, $A = \mathbb Z[t,Y]$ are polynomial rings with $q=t$ and $y = Y$.
But then, we can even assume that $R = \mathbb Q(t)$ and $A = \mathbb Q(t)[Y]$.
In particular, we are in a situation where all positive $q$-integers are invertible in $A$.
Then the map \eqref{dppol} becomes bijective.
Moreover, using lemma \ref{funmul}, we see that that the multiplication on both sides coincide because
\[
 (m)_{q}!(n)_{q}! {m + n -i \choose m}_{q}{m \choose i}_{q} = (m+n-i)_{q}!   (i)_{q}! {m \choose i}_{q}{n \choose i}_{q}
\]
as one easily checks.

Finally, assume that $n > k$.
Then, for $i \leq \min (m, n)$, we have $i \leq m$ and therefore $m+n-i \geq n > k$.
It follows that $\xi^{[m]}\xi^{[n]} \equiv 0 \mod I^{[k]}$, and $I^{[k]}$ is an ideal.
\end{proof}

\begin{xmp}
\begin{enumerate}
\item For all $k \in \mathbb N$, we have
\[
\xi^{[k]} \xi = (k+1)_{q}\xi^{[k+1]} - (k)_{q}y\xi^{[k]}.
\]
\item We have
\[
(\xi^{[2]})^2 = (2)_{q^2}(3)_{q}\xi^{[4]} - (3)_{q}(2)_{q} y\xi^{[3]} + qy^2\xi^{[2]}.
\]
\end{enumerate}
\end{xmp}

\begin{dfn} \label{dfdp}
The free $A$-module $A\langle\xi\rangle$ on the (abstract) generators $\xi^{[n]}$ with $n \in \mathbb N$, endowed with the multiplication rule of proposition \ref{dpring}, is the \emph{twisted divided power polynomial ring} over $A$.
\end{dfn}

\begin{rmk}
\begin{enumerate}
\item
It is important to remind that $q$ and $y$ are built into this definition.
As already mentioned, if we want to make clear the dependence on the parameters, we will write $A\langle\xi\rangle_{q,y}$.
\item
The coefficients in the multiplication formula \eqref{sqform} are polynomials in $q$.
Actually if we consider the map $\mathbb Z[t][Y] \to A$ that sends $t$ to $q$ and $Y$ to $y$, there exists an isomorphism of $A$-algebras
\[
A \otimes_{\mathbb Z[t][Y]} \mathbb Z[t][Y]\langle \xi \rangle \simeq A \langle \xi \rangle.
\]
\item
The filtration of $A\langle\xi\rangle$ by the ideals $I^{[n+1]}$ will be called the \emph{divided power filtration} or \emph{ideal} filtration.
Note that $A\langle\langle\xi\rangle\rangle$ inherits the structure of a commutative $A$-algebra.
\end{enumerate}
\end{rmk}

\begin{xmp}
\begin{enumerate}
\item In the case $q = 1$ and $y = 0$, we fall back onto the usual divided power polynomial ring.
\item When $q \neq 1$ but still $y = 0$, is is possible to develop a general theory of $q$-divided powers, and $A\langle\xi\rangle$ will be the divided power polynomial ring for this theory.
We do not know how to achieve this in general.
\item Assume $R=A = \mathbb F_{2}$, $q=1$ and $y = 1$.
In this situation, we have $\xi^2 = \xi$ in $A\langle\xi\rangle$ but there exists no non trivial idempotent of degree 1 in the usual divided power polynomial ring.
Thus we see that when $q = 1$ but $y \neq 0$, the ring $A\langle\xi\rangle$ is \emph{not} isomorphic to the usual divided power polynomial ring.
\end{enumerate}
\end{xmp}

\begin{lem} \label{freeb}
Assume $R$ is $q$-divisible of $q$-characteristic $p > 0$.
Then, the ideal generated by $\xi$ in $A\langle\xi\rangle$ is the free $A$-module generated by all $\xi^{[k]}$ with $p \nmid k$.
\end{lem}

\begin{proof}
The formulas
\begin{equation} \label{kindu}
\forall k \in \mathbb N, \quad \xi^{[k]} \xi = (k+1)_{q}\xi^{[k+1]} + (k)_{q}y\xi^{[k]}
\end{equation}
show that  the ideal $A\langle\xi\rangle\xi$ is contained in the $A$-module generated by all $(k)_{q}\xi^{[k]}$'s.
Since $(k)_{q} = 0$ when $p \mid k$, we see that $A\langle\xi\rangle\xi$ is actually contained in the free $A$-module generated by all $\xi^{[k]}$'s with $p \nmid k$.
Conversely, formula \eqref{kindu} also tells us that
\[
(k+1)_{q}\xi^{[k+1]} \equiv  (k)_{q}y\xi^{[k]} \mod \xi
\]
for all $k$.
Using the fact that we always have $(kp+i)_{q} = (i)_{q}$, we see that for all $k \in \mathbb N$, we have
\[
\xi^{[kp+1]} = (kp+1)_{q} \xi^{[kp+1]} \equiv  (kp)_{q}y\xi^{[kp]} = 0 \mod \xi
\]
Then, by induction on $i$, we get for $1 < i < p$,
\[
(i)_{q}\xi^{[kp +i]} \equiv  (i-1)_{q}y\xi^{[kp+i-1]} \equiv 0 \mod \xi
\]
and we easily conclude since $(i)_{q} \in R^\times$ for $0 < i < p$ because $R$ is $q$-divisible.
\end{proof}

\begin{dfn}
Assume that $q\mathrm{-char}(R) = p > 0$.
Then the unique $A$-linear map
\begin{equation} \label{callu}
\xymatrix@R0cm{A\langle\omega\rangle_{1,y^p} \ar[r] & A\langle\xi\rangle_{q,y} \\
\omega^{[k]} \ar@{|->}[r] & \xi^{[kp]}  }
\end{equation}
is the \emph{twisted divided $p$-power map}.
\end{dfn}

\begin{rmk}
We will not need it but it should be noticed that when $p$ is \emph{not} the $q$-characteristic of $R$, the definition has to be modified a little bit:
the \emph{twisted divided power map} will be given by
\[
\xymatrix@R0cm{A\langle\omega\rangle_{q^p,y^p} \ar[r] & A\langle\xi\rangle_{q,y}
\\\omega^{[k]} \ar@{|->}[r] & \prod_{i=2}^k {ip-1 \choose p-1} \xi^{[kp]}.  }
\]
\end{rmk}

\begin{thm} \label{dualpc}
Assume that $q\mathrm{-char}(R) = p > 0$.
If $R$ is $q$-flat, then the twisted divided power map is a ring homomorphism.
If $R$ is $q$-divisible, then it induces an isomorphism of $A$-algebras
\begin{equation}\label{dupec}
A\langle\omega\rangle_{1,y^p} \simeq A\langle\xi\rangle_{q,y}/(\xi).
\end{equation}
\end{thm}

Recall that the first condition means that $q$ is a primitive $p$th root of unity or that $q=1$ and $R$ has positive characteristic $p$.
Moreover, $q$-divisibility is satisfied if $p$ is prime or if $q$ belongs to a subfield $K$ of $R$ for example.

\begin{proof}
By definition, if we denote by $u$ the twisted divided power map \eqref{callu}, we have
\[
u(\omega^{[k]}) = \xi^{[kp]}.
\]
Therefore, it follows from lemma \ref{freeb} that the map \eqref{dupec} is an isomorphism of $A$-modules when $R$ is $q$-divisible.
Thus, it only remains to show that $u$ is a ring homomorphism when $R$ is $q$-flat.
In other words, we want to check that
\begin{equation} \label{frobco}
\forall k, l \in \mathbb N, \quad u(\omega^{[k]} \omega^{[l]}) = u(\omega^{[k]})u (\omega^{[l]}).
\end{equation}
Since
\[
\omega^{[k]} \omega^{[l]} = \sum_{i=0}^{\min{(k,l)}} (-1)^i{k + l -i \choose k}{k \choose i} y^{ip}\omega^{[k+l-i]},
\]
the left hand side of equality \eqref{frobco} is equal to
\[
 \sum_{i=0}^{\min{(k,l)}} (-1)^{i}{k + l -i \choose k}{k \choose i} y^{ip}\xi^{[kp+lp-ip]}.
\]
We can also compute the right hand side
\[
\xi^{[kp]} \xi^{[lp]} = \sum_{i=0}^{\min{(kp,lp)}} (-1)^iq^{\frac{i(i-1)}2}{kp + lp -i \choose kp}_{q}{kp \choose i}_{q} y^i\xi^{[pk+pl-i]}.
\]
Our assertion therefore follows from the twisted Lucas theorem (proposition 2.13 of \cite{LeStumQuiros15}) thanks to lemma \ref{minone} below.
\end{proof}

\begin{lem} \label{minone}
Assume that $p := q\mathrm{-char}(R) > 0$ and $R$ is $q$-flat.
Then
\[
(-1)^{ip} = (-1)^iq^{\frac {ip(ip-1)}2} 
\]
\end{lem}

\begin{proof}
If $p$ is odd, then either $i$ is even or $ip$ is odd and we may therefore write
\[
q^{\frac {ip(ip-1)}2} = (q^p)^{\frac {i(ip-1)}2} = 1
\]
because $q^p = 1$.
Now one easily sees that $(-1)^{ip} = ((-1)^p)^i = (-1)^i$.

If we assume that $p$ is even so that $p = 2k$ with $k \in \mathbb N$, then we know from proposition 1.11 of \cite{LeStumQuiros15} that, since $R$ is $q$-flat, we have $q^k = -1$ and the formula also holds.\end{proof}

We want to consider now the paring of $A$-modules
\[
<\ , \ >\ : A[\theta] \times A\langle\omega\rangle \to A
\]
given by
\[
\forall m, n \in \mathbb N, \quad <\theta^{m}, \omega^{[n]} >\ = \left\{ \begin{array} {l}1 \ \mathrm{if}\ n = m \\ 0\ \mathrm{otherwise.} \end{array}\right.
\]
Strictly speaking, this is not a perfect paring.
However, it induces for each $n \in \mathbb N$, a perfect pairing between the $A$-submodule (or quotient)
\[
A[\theta]_{\leq n} \simeq A[\theta]/\theta^{n+1}
\]
of polynomials of degree at most $n$ and  the $A$-submodule (or quotient)
\[
A\langle\omega\rangle_{\leq n} \simeq A\langle\omega\rangle/I^{[n+1]}
\]
of twisted divided power polynomials of degree at most $n$. 
Alternatively, we can say that it induces perfect parings between $A[[\theta]]$ and $A\langle\omega\rangle$ as well as between $A[\theta]$ and $A\langle\langle\omega\rangle\rangle$.

\begin{prop} \label{dualcom}
Assume that $q = 1$.
Then,
\begin{enumerate}
\item multiplication on $A[\theta]$ is dual to the morphism of $A$-algebras
\begin{equation} \label{duna}
\xymatrix@R0cm{  A\langle\omega\rangle \ar[r]^-\delta & A\langle\omega\rangle\otimes A\langle\omega\rangle
\\  \omega^{[n]} \ar@{|->}[r] & \sum_{i=0}^n \omega^{[i]} \otimes \omega^{[n-i]}.}
\end{equation}
\item multiplication on $A\langle \omega \rangle$ is dual to the morphism of $A$-algebras
\[
\xymatrix@R0cm{  A[\theta] \ar[r] & A[\theta] \otimes A[\theta]
\\  \theta \ar@{|->}[r] & 1 \otimes \theta + \theta \otimes 1 - y \theta \otimes \theta.
}
\]
\end{enumerate}
\end{prop}

\begin{proof}
We essentially use the fact that the $\theta^n$'s and the $\omega^{[n]}$'s become dual basis under our pairing and that the dual to a matrix is its transpose.

Since multiplication on the polynomial ring $A[\theta]$ is given by
\[
\theta^m\theta^n = \theta^{m+n} = \sum_{m+n=k} \theta^k,
\]
comultiplication on $A\langle\omega\rangle$ will be given by
\[
\omega^{[k]} \mapsto \sum_{m+n=k} \omega^{[m]} \otimes \omega^{[n]}
\]
and changing indices ($k$ becomes $n$, $m$ becomes $i$ and therefore $n = k - m$ has to be turned into $n-i$) will give what we want.

We also have to show that this comultiplication map is a ring morphism.
As usual, we may assume that all the non zero integers are invertible.
We may then refer to lemma \ref{binai} which identifies the morphism \eqref{duna} with the morphism \eqref{delna}.

We proceed in the same way for the second assertion.
Multiplication on $A\langle\omega\rangle$ is given by
\[
\omega^{[m]} \omega^{[n]} = \sum_{m+n-i=k} (-1)^i{k \choose m}{m \choose i} y^i\omega^{[k]}
\]
and comultiplication will therefore be given by
\begin{equation} \label{formdoub}
\theta^k \mapsto \sum_{m+n-i=k} (-1)^i{k \choose m}{m \choose i} y^i \theta^m \otimes \theta^n.
\end{equation}
On the other hand, we have
\begin{eqnarray*}
(1 \otimes \theta + \theta \otimes 1 - y \theta \otimes \theta)^k &=& \sum_{i \leq j \leq k} {k \choose j}{j \choose i} (1 \otimes \theta)^{k-j}(\theta \otimes 1)^{j-i}(-y\theta \otimes \theta)^i
\\
&=& \sum_{i \leq j \leq k} {k \choose j}{j \choose i} (-1)^i y^i \theta^j \otimes \theta^{k-j+i}
\end{eqnarray*}
which is exactly the same as \eqref{formdoub} (up to the renaming of $m$ into $j$).
\end{proof}

\section{Twisted divided powers and twisted algebras}

We assume now that $A$ is a \emph{twisted commutative $R$-algebra} (a commutative $R$-algebra endowed with an $R$-linear ring endomorphism $\sigma_{A}$) and that $\sigma_{A}(y) = qy$.
We will investigate the relation of $\sigma_{A}$ with twisted divided powers relative to $q$ and $y$.

We endow the polynomial ring $A[\xi]$ with the unique $\sigma_{A}$-linear endomorphism such that
\[
\sigma_{A,y}(\xi) = \xi + y.
\]

In practice, we will usually write $\sigma$ instead of $\sigma_{A}$ or $\sigma_{A,y}$ in order to make the notations lighter.

\begin{prop} \label{itxi}
We have in $A[\xi]$,
\begin{equation} \label{itsig}
\forall n \in \mathbb N, \quad \sigma^n(\xi) = \xi + (n)_{q}y.
\end{equation}
Actually, if $\sigma$ is bijective on $A$ and $q \in R^\times$, then $\sigma$ is bijective on $A[\xi]$ and formula \eqref{itsig} holds for any $n \in \mathbb Z$.
\end{prop}

\begin{proof}
By induction, we will have for all $n \in \mathbb N$,
\[
\sigma^n(\xi) = \sigma (\xi + (n-1)_{q}y) = (\xi + y) + (n-1)_{q}qy = \xi +  ( 1 + q(n-1)_{q}) y
\]
and we know that $1 + q(n-1)_{q} = (n)_{q}$.

Assume that $\sigma$ is bijective on $A$ and $q \in R^\times$.
Then, from $\sigma(y) = qy$, we get $\sigma^{-1}(y) = q^{-1}y$.
If moreover, $\sigma$ is bijective on $A[\xi]$, then we deduce from the equality $\sigma(\xi) = \xi + y$ that
\[
\xi = \sigma^{-1}(\xi + y) = \sigma^{-1}(\xi) + \sigma^{-1}(y) = \sigma^{-1}(\xi) + q^{-1}y
\]
and it follows that
\[
\sigma^{-1}(\xi) = \xi -q^{-1}y.
\]
Conversely, this formula can be used to define an inverse to $\sigma$ on $A[\xi]$.
Finally, applying this to $\sigma^n$ (and therefore replacing $y$ by $(n)_{q}y$ and $q$ by $q^n$), we obtain as claimed:
\[
\sigma^{-n}(\xi) = \xi -q^{-n}(n)_{q}y = \xi + (-n)_{q}y.\qedhere
\]
\end{proof}

\begin{rmk}
\begin{enumerate}
\item
As a consequence of the proposition, we see that if $q\mathrm{-char}(R) = p > 0$, then $\sigma^p(\xi) = \xi$ (and of course, also $\sigma^p(y) = y$).
\item
As usual, most formulas will be polynomial in $q$, $y$ and $\xi$.
More precisely, we may usually reduce to the case $R = \mathbb Z[t]$ (and often to $R = \mathbb Q(t)$) and $q = t$.
In other words, we would work in $\mathbb Z[t,Y,\xi]$ with $\sigma(t) = t$, $\sigma(Y) = tY$ and $\sigma(\xi) = \xi + Y$.
\end{enumerate}
\end{rmk}

Recall that we defined in section 4 of \cite{LeStumQuiros15} the \emph{twisted powers} of $f \in A[\xi]$ with respect to $\sigma$ as
\[
f^{(n)_{\sigma}} = f\sigma(f) \cdots \sigma^{n-1}(f).
\]

\begin{cor} We have
\[
\forall n \in \mathbb N, \quad y^{(n)_{\sigma}} = q^{\frac{n(n-1)}2}y^n \quad \mathrm{and} \quad \xi^{(n)_{\sigma}} = \xi^{(n)_{q,y}} := \prod_{i=0}^{n-1}(\xi + (i)_{q}y).
\]
\end{cor}

\begin{proof}
Immediately follows from the condition $\sigma(y) = qy$ and proposition \ref{itxi}.
\end{proof}

We will drop the index $\sigma$ when we believe that no confusion will arise (in particular, this is consistent with the notations of the previous section).
But we might also write $y^{(n)_{q}}$ and $\xi^{(n)_{q,y}}$ respectively if we want to insist on the choice of $q$ and $y$.

We will need below the following formula:

\begin{lem}\label{funmul2}
In $A[\xi]$, we have for all $n \in \mathbb N$,
\[
\sigma(\xi^{(n)})  = \sum_{i=0}^{n} (i)_{q}! {n \choose i}_{q} y^{i}\xi^{(n-i)}.
\]
\end{lem}

\begin{proof}
By induction, we will have
\begin{eqnarray*}
\sigma(\xi^{(n)}) & = & \sigma(\xi^{(n-1)})\sigma^n(\xi)
\\
 & = & \sigma(\xi^{(n-1)})(\xi + (n)_{q} y)
 \\
 & = & \xi\sigma(\xi^{n-1}) + (n)_{q} y \sigma(\xi^{(n-1)})
\\
 & = & \xi^{(n)} + (n)_{q} y \sum_{i=0}^{n-1} (i)_{q}! {n-1 \choose i}_{q} y^{i}\xi^{(n-1-i)}
\\
 & = & \xi^{(n)} +  \sum_{i=1}^{n}(n)_{q} (i-1)_{q}! {n-1 \choose i-1}_{q} y^{i}\xi^{(n-i)}
\end{eqnarray*}
and the result follows from the identity
\[
(n)_{q} (i-1)_{q}! {n-1 \choose i-1}_{q} = (i)_{q}! {n \choose i}_{q}. \qedhere
\]\end{proof}

\begin{prop}
The unique $\sigma$-linear endomorphism of $A\langle\xi\rangle$ such that
\[
\forall n \in \mathbb N, \quad \sigma(\xi^{[n]})  = \sum_{i=0}^{n} y^{i}\xi^{[n-i]},
\]
is a ring homomorphism.
Moreover, the map \eqref{dppol} is a morphism of twisted $R$-algebras.
\end{prop}

Recall that a \emph{morphism of twisted rings} (or algebras) is a morphism which commutes with the given endomorphisms.

\begin{proof}
As we did several times in section \ref{first}), we can easily reduce to the case of $R = \mathbb Q(t)$ and $q = t$ and we may therefore assume all $q$-integers are invertible in $R$.
Then the map \eqref{dppol} becomes bijective.
We may then use lemma \ref{funmul2} and the equality
\[
(n-i)_{q}!(i)_{q}! {n \choose i}_{q} = (n)_{q}!.\qedhere
\]\end{proof}

Again, if necessary, we will write $\sigma_{q,y}$ to make clear the dependence in $q$ and $y$.

\begin{rmk}
\begin{enumerate}
\item
The endomorphism $\sigma$ of $A\langle \xi \rangle$ is \emph{not} continuous and does not extend to a ring endomorphism of $A\langle\langle\xi\rangle\rangle$.
\item
We have to be careful that, in general, $\sigma^p$ will not  be the identity on $A\langle\xi\rangle$ even if it is so on $A[\xi]$.
For example, if $q =-1$, we will have
\[
\sigma(\xi) = \xi + y \quad \mathrm{and} \quad \sigma(\xi^{[2]})  =\xi^{[2]} + y \xi + y^2,
\]
and therefore
\[
\sigma^2(\xi^{[2]}) = \sigma(\xi^{[2]}) - y \sigma(\xi) + y^2 = \xi^{[2]} + y \xi + y^2 - y(\xi +y) + y^2
= \xi^{[2]} + y^2.
\]
\end{enumerate}
\end{rmk}

Actually, we can give a general formula for the powers of $\sigma$ on $A\langle\xi\rangle$:

\begin{prop} \label{itdiv}
We have
\[
\forall p \in \mathbb N, \forall n \in \mathbb N, \quad \sigma^{p}(\xi^{[n]})  = \sum_{i=0}^{n} {p+i -1\choose i}_{q}y^{i}\xi^{[n-i]}.
\]
\end{prop}

\begin{proof}
By induction, we will have
\begin{eqnarray*}
\sigma^{p+1}(\xi^{[n]})  &=& \sigma\left(\sum_{k=0}^{n} {p+k -1\choose k}_{q}y^{k}\xi^{[n-k]} \right)
\\
&=& \sum_{k=0}^{n} {p+k -1\choose k}_{q}\sigma(y^{k})\sigma(\xi^{[n-k]}) 
\\
&=& \sum_{k=0}^{n} {p+k -1\choose k}_{q}q^ky^k \left(\sum_{j=0}^{n-k} y^{j}\xi^{[n-k-j]}\right)\\
&=& \sum_{i=0}^{n} \left(\sum_{k=0}^{i} {p+k -1\choose k}_{q}q^{k} \right)y^i \xi^{[n-i]}.
\end{eqnarray*}
In order to get the formula, is is sufficient to notice that, by definition (and induction), we have
\[
\sum_{k=0}^{i} {p+k -1\choose k}_{q}q^{k} = {p+i \choose i}_{q}.\qedhere
\]\end{proof}

The multiplication rule is quite involved in $A\langle\xi\rangle$ but the twisted multiplication is much simpler:

\begin{prop}
We have
\[
\forall n, m \in \mathbb N, \quad \xi^{[n]}\sigma^n(\xi^{[m]}) =  {m+n \choose n}_{q} \xi^{[n+m]}
\]
\end{prop}

\begin{proof}
We may assume that all $q$-integers are invertible in $R$ and use assertion 1) of lemma 4.3 of \cite{LeStumQuiros15} which gives $\xi^{(n)}\sigma^n(\xi^{(m)}) =  \xi^{(n+m)}$.
\end{proof}

Given any natural integer $p$, we have
\[
\sigma(y^p) = \sigma(y)^p = (qy)^p = q^py^p.
\]
We may therefore also apply all the above considerations to the situation $q^p \in R$ and $y^p \in A$, and consider the twisted $R$-algebra $A\langle\omega\rangle_{q^p, y^p}$.
In the particular case $q\mathrm{-char}(R) = p$, we fall onto $A\langle\omega\rangle_{1, y^p}$.
Recall that the twisted divided power map induces an isomorphism of $A$-algebras
\[
A\langle\omega\rangle_{1,y^p} \simeq A\langle\xi\rangle_{q,y}/(\xi).
\]
when $R$ is $q$-divisible of positive $q$-characteristic $p$.

\begin{prop}
Assume $R$ is $q$-divisible of $q$-characteristic $p > 0$.
Then, the canonical map
\begin{equation}\label{dupec2}
A\langle\xi\rangle_{q,y} \to A\langle\xi\rangle_{q,y}/(\xi) \simeq A\langle\omega\rangle_{1,y^p}
\end{equation}
is a morphism of twisted $A$-algebras.
\end{prop}

\begin{proof}
If we denote by $u$ the twisted divided $p$-power map \eqref{callu}, we need to check that
\begin{equation} \label{frsig}
\forall k \in \mathbb N, \quad (u \circ \sigma)(\omega^{[k]}) \equiv (\sigma \circ u)(\omega^{[k]}) \quad \mod \xi.
\end{equation}
From lemma \ref{funmul2}, we know that
\[
\sigma(\omega^{[k]})  = \sum_{i=0}^{k} y^{ip}\omega^{[k-i]}
\]
and it follows that
\[
(u \circ \sigma)(\omega^{[k]}) = \sum_{i=0}^{k} y^{ip}\xi^{[kp-ip]}
\]
On the other hand, using lemma \ref{funmul2} again, we have
\[
(\sigma \circ u)(\omega^{[k]}) = \sigma(\xi^{[kp]})  = \sum_{i=0}^{kp} y^{i}\xi^{[kp-i]}
\]
and we are done thanks to lemma \ref{freeb}.
\end{proof}

The next result is interesting mostly in the case $q = 1$ and we will therefore use $\omega$ instead of $\xi$.

\begin{prop} \label{dualcom2}
We have
\[
\forall f \in A[\theta], \forall g \in A\langle\omega\rangle, \quad < (1 -y\theta) f , \sigma(g) > = \sigma(< f, g>)
\]
\end{prop}

\begin{proof}
By $\sigma$-linearity, it is sufficient to compute for $m, n \in \mathbb N$,
\begin{eqnarray*}
 < (1 -y\theta) \theta^m , \sigma(\omega^{[n]}) >\ &= & < \theta^m -y\theta^{m+1} , \sum_{i=0}^ny^i\omega^{[n-i]} >
\\
& = &  \sum_{i=0}^n y^i< \theta^m , \omega^{[n-i]} >\ -  \sum_{i=0}^n y^{i+1}<  \theta^{m+1} ,\omega^{[n-i]} >
\\
&=& < \theta^m , \omega^{[n]} >\ + \sum_{i=1}^n y^i< \theta^m , \omega^{[n-i]} >\ 
\\
&& -  \sum_{i=0}^{n-1} y^{i+1}<  \theta^{m+1} ,\omega^{[n-i]} >\ - y^{n+1}<  \theta^{m+1} , 1>.
\end{eqnarray*}
The middle sums cancel each other and the last term is $0$.
\end{proof}

\begin{rmk}
We may also wonder about the dual (for the above pairing) to the endomorphism $\sigma$ of $A\langle\omega\rangle$ when $\sigma_{A} = \mathrm{Id}_{A}$.
We just transform
\[
\sigma(\omega^{[n]})  = \sum_{i=0}^{n} y^{i}\omega^{[n-i]} = \sum_{i+j=n} y^{i}\omega^{[j]}
\]
to its dual formula and get
\[
\theta^{j} \mapsto \sum_{i+j=n} y^{i} \theta^n = \sum_{i=0}^{\infty} y^{i} \theta^{i+j}
\]
which shows that we must introduce power series.
More precisely, writing $n$ instead of $j$, we obtain
\[
\theta^{n} \mapsto \sum_{i=0}^{\infty} y^{i}\theta^{i+n} = (\sum_{i=0}^\infty y^i \theta^i)\theta^n
\]

In other words, we see that the dual to $\sigma$ on $A\langle\omega\rangle$ is exactly division by $1 - y\theta \in A[[\theta]]^\times$:
\[
\xymatrix@R0cm{  A[[\theta]] \ar[r] & A[[\theta]]
\\  f(\theta) \ar@{|->}[r] & \displaystyle \frac {f(\theta)}{1 -y\theta}.}
\]
This is not a surprise according to proposition \ref{dualcom2}.
Of course, in order to define this map, we may as well work over the localized ring $A[\theta, \frac 1 {1 -y\theta}]$.
This map is \emph{not} a ring homomorphism.
\end{rmk}

\section{Twisted principal parts of level zero}

We assume now that $A$ is a \emph{twisted commutative algebra}.
It means that $A$ is a twisted commutative algebra but we also assume that there exists a \emph{twisted coordinate} (we recall below what it means) $x \in A$ such that $\sigma(x) = qx + h$ with $q,h \in R$.
We set $y := x -\sigma(x)$.

In order to apply the results of the previous section, we need to check the following:

\begin{lem}
In the ring $A$, we have $\sigma(y) = qy$.
\end{lem}

\begin{proof}
We have $y = (1 -q)x - h$ and therefore
\[
\sigma(y) = (1 -q)(qx +h) - h = q(1-q)x + qh = qy.\qedhere
\]\end{proof}

As we did before, we will endow the polynomial ring $A[\xi]$ with the unique $\sigma$-linear ring endomorphism such that
\[
\sigma(\xi) = \xi + y.
\]

We now review some material from \cite{LeStumQuiros18}.
We endow $\mathrm P := A \otimes_{R} A$ with the endomorphism $\sigma_{\mathrm P} := \sigma_{A} \otimes \mathrm{Id}_{A}$.
We will always see $\mathrm P$ as an $A$-module via the action on the left and simply write $z := z \otimes 1 \in \mathrm P$ when $z  \in A$.
By contrast, we set $\tilde z := 1 \otimes z$.
We will also write the morphism giving the right action as
\[
\xymatrix@R0cm{\Theta :& A \ar[r] & \mathrm P
\\ & z \ar@{|->}[r] & \tilde z. }
\]
We denote by $I \subset \mathrm P$ the kernel of multiplication on $A$ and consider the modules of \emph{twisted principal parts of infinite level} $\mathrm P^{(\infty)}_{(n)_\sigma} := \mathrm P/I^{(n+1)_{\sigma}}$ with
\[
I^{(n)_{\sigma}} = I\sigma(I)\cdots \sigma^{n-1}(I)
\]

There exists a unique morphism of twisted $R$-algebras
\begin{equation} \label{maplk}
\xymatrix@R0cm{  A[\xi] \ar[r] & \mathrm P
\\ \xi \ar@{|->}[r] & \tilde x - x.}
\end{equation}
We assumed above that $x$ is a \emph{twisted coordinate} on $A$: it means that the map \eqref{maplk} induces an isomorphism
\[
A[\xi]/\xi^{(n+1)} \simeq \mathrm P^{(\infty)}_{(n)_{\sigma}}
\]
for all $n \in \mathbb N$.
We may then see $A[\xi]$ as a subring of the ring $\widehat {\mathrm P}_{\sigma}^{(\infty)} := \varprojlim  \mathrm P^{(\infty)}_{(n)_{\sigma}}$ of \emph{twisted principal parts of infinite order}.
We might index all these objects with $A/R$ if we want to make clear the dependence on $A$ and $R$.

\begin{dfn}
The $A$-module of \emph{twisted principal parts of order at most $n$ and level $0$} of $A$ is
\[
\mathrm P_{A/R,(n)_{\sigma}}^{(0)} := A\langle\xi\rangle/I^{[n+1]}
\]
And the $A$-module of \emph{twisted principal parts of infinite order and level $0$} of $A$ is
\[
\widehat {\mathrm P}_{A/R,\sigma}^{(0)} := \varprojlim \mathrm P_{A/R,(n)_{\sigma}}^{(0)} (= A\langle\langle\xi\rangle\rangle).
\]
\end{dfn}

In order to lighten the notations, we will sometimes drop the index $A/R$.

\begin{rmk}
\begin{enumerate}
\item
Unlike the infinite level analog, this notion depends on $q$ and $x$ and not only on $\sigma$.
\item
If we still denote by $X$ an indeterminate, then there exists a canonical $A$-linear isomorphism of rings
\[
A \otimes_{\mathbb Z[t,X]} \mathrm P_{\mathbb Z[t,X]/\mathbb Z[t],(n)_{\sigma}}^{(0)} \simeq \mathrm P_{A/R,(n)_{\sigma}}^{(0)}.
\]
\item
We might also have to consider the intermediate and completed ideals
\[
I^{[k]}_{A/R,(n)_{\sigma}} := I^{[k]}/I^{[n+1]} \quad \mathrm{and} \quad \widehat I^{[k]}_{A/R,\sigma} = \varprojlim I^{[k]}_{A/R,(n)_{\sigma}}.
\]
for $k \leq n+1$.
\item
By definition, $\mathrm P^{(0)}_{(n)_{\sigma}}$ is the finite free $A$-module on the images of the $\xi^{[i]}$ for $i \leq n$ and $I^{[k]}_{(n)_{\sigma}}$
is the free $A$-module on the images of the $\xi^{[i]}$ for $k \leq i \leq n$.
It follows that $\widehat {\mathrm P}_{A/R,\sigma}^{(0)}$ (resp. $\widehat I^{[k]}_{A/R}$) is the set of infinite sums $\sum z_{i} \xi^{[i]}$ with $z_{i} \in A$ and $i \in \mathbb N$ (resp. and $i \geq k$).
\item
Formula \eqref{sqform} shows that $I^{[n+1]}$ is an ideal inside $A\langle\xi\rangle$.
It follows that the quotients $\mathrm P_{A/R,(n)_{\sigma}}^{(0)}$ have a natural structure of $A$-algebra and so does $\widehat {\mathrm P}_{A/R,\sigma}^{(0)}$.
\end{enumerate}
\end{rmk}

\begin{lem} \label{bijag}
The map \eqref{dppol} sends $\xi^{(m)}$ inside $I^{[n+1]}$ when $m > n$.
In particular, it induces a homomorphism
\[
A[\xi]/\xi^{(n+1)} \to \mathrm P^{(0)}_{(n)}.
\]
This is even an isomorphism if all $(m)_{q}$ are invertible in $R$ for $m \leq n$.
\end{lem}

\begin{proof}
Same arguments as for lemma \ref{bijfond}.
\end{proof}

\begin{prop} \label{compinf}
There are canonical homomorphisms
\[
\mathrm P^{(\infty)}_{A/R,(n)_{\sigma}} \to \mathrm P_{A/R,(n)_{\sigma}}^{(0)} \quad \mathrm{and} \quad \widehat {\mathrm P}^{(\infty)}_{A/R,\sigma} \to  \widehat {\mathrm P}_{A/R,\sigma}^{(0)}
\]
which are bijective when all $(m)_{q}$ are invertible (for $m \leq n$ in the first case).
\end{prop}

\begin{proof}
Follows from lemmas \ref{bijfond} and \ref{bijag}.
\end{proof}

When this last condition is satisfied, we might identify both rings and drop the superscript, writing simply $\mathrm P_{A/R,(n)_{\sigma}}$ or $\widehat {\mathrm P}_{A/R,\sigma}$.
\begin{dfn}
The \emph{twisted Taylor map of level zero} is the composite homomorphism of $R$-algebras
\[
\xymatrix{ A \ar[drr]_{\widehat \Theta^{(0)}} \ar[r]^{\Theta} \ar@/^2pc/[rr]^{\widehat \Theta^{(\infty)}} & \mathrm P_{A/R} \ar[r] & \widehat {\mathrm P}^{(\infty)}_{A/R,\sigma} \ar[d] 
\\ && \widehat {\mathrm P}_{A/R,\sigma}^{(0)}&  .
}
\]
\end{dfn}

Also, we will denote by
\[
\Theta_{n}^{(0)} : A  \to \mathrm P^{(0)}_{A/R,(n)_{\sigma}}
\]
the composition of the twisted Taylor map and the projection.
When there is no risk of confusion, we might simply write $\Theta$ for any of these Taylor maps.

We can give an explicit expression for the twisted Taylor map as we shall see shortly.
First of all, since $x$ is a twisted coordinate on $A$, we know from proposition 2.10 of \cite{LeStumQuiros18} that there exists a unique $R$-linear endomorphism $\partial_{\sigma,A}$ of $A$ such that
\[
\forall z_{1}, z_{2} \in A, \quad \partial_{\sigma,A}(z_{1}z_{2}) = z_{1}\partial_{\sigma,A}(z_{2}) + \sigma(z_{2}) \partial_{\sigma,A}(z_{1})
\]
(a \emph{$\sigma$-derivation}) and $\partial_{\sigma,A}(x) = 1$.
We will often simply write $\partial_{\sigma}$, but this endomorphism should not be confused with the abstract generator of the twisted Weyl algebra that we will denote later in the same way.

\begin{prop}
We have
\[
\forall z \in A, \quad \widehat\Theta^{(0)}(z) = \sum_{k=0}^\infty \partial_{\sigma}^k(z) \xi^{[k]}.
\]
\end{prop}

\begin{proof}
Recall from proposition 5.5 of \cite{LeStumQuiros18} that there exists a family of endomorphisms $\partial_{\sigma}^{[k]}$ of $A$ such that
\[
\forall z \in A, \quad \widehat\Theta^{(\infty)} (z) = \sum_{k=0}^\infty \partial_{\sigma}^{[k]}(z) \xi^{k}.
\]
The proposition then follows from corollary 6.2 of \cite{LeStumQuiros18} where we showed that
\[\forall k \in \mathbb N, \forall z \in A, \quad \partial_{\sigma}^k(z) = (k)_{q}! \partial_{\sigma}^{[k]}(z).\qedhere
\]\end{proof}

\begin{xmp}
\begin{enumerate}
\item We always have $\Theta(x) = x + \xi$.
\item We have
\[
\Theta(x^2) = x^2 + ((1+q)x+h)) \xi + (1+q)\xi^{[2]}.
\]
\item If $\sigma(x) = qx$ with $q \in R^\times$ and $x \in A^\times$, one can show that
\[
\Theta\left(\frac 1x\right) = \sum_{k=0}^{\infty} (-1)^k \frac  {(k)_{q}!\xi^{[k]}}{ q^{\frac {k(k+1)}2} x^{k+1}} = \frac 1x - \frac \xi{x^2} + \frac {(1+q)\xi^{[2]}}{q^3x^3} - \cdots.
\]
\end{enumerate}
\end{xmp}

We will denote by
\[
A^{\partial_{\sigma}=0} := \mathrm H^0_{\partial_{\sigma}}(A) = \{z \in A, \quad \partial_{\sigma}(z) = 0\}
\]
the subalgebra of \emph{horizontal sections} of $A$.

\begin{prop} \label{leseq}
There exists a left exact sequence
\[
\xymatrix{A^{\partial_{\sigma}=0} \ar[r] & A \ar@<2pt>[rr]^-{\mathrm{can}} \ar@<-2pt>[rr]_-{\Theta} && \widehat {\mathrm P}_{\sigma}^{(0)}.}
\]
\end{prop}

\begin{proof}
We have $\Theta(z) = \sum \partial_{\sigma}^k(z) \xi^{[k]}$ and it follows that $\Theta(z) = z$ if and only if $\partial_{\sigma}(z) = 0$. \end{proof}

\begin{prop} \label{epimono}
There exists an epi-mono factorization
\[
\mathrm P_{A/R} \twoheadrightarrow A \otimes_{A^{\partial_{\sigma}= 0}} A \hookrightarrow  \widehat {\mathrm P}_{A/R}^{(0)} .
\]
When $R$ is $q$-divisible of $q$-characteristic $p > 0$, there exists another epi-mono factorization
\[
\widehat {\mathrm P}^{(\infty)}_{A/R,\sigma} \twoheadrightarrow \mathrm P^{(\infty)}_{A/R,(p-1)_{\sigma}} \hookrightarrow \widehat {\mathrm P}^{(0)}_{A/R,\sigma}.
\]
\end{prop}

As a consequence, when $R$ is $q$-divisible of $q$-characteristic $p > 0$, we obtain an inclusion
\[
A \otimes_{A^{\partial_{\sigma}= 0}} A \hookrightarrow \mathrm P^{(\infty)}_{A/R,(p-1)_{\sigma}} \simeq \mathrm P^{(0)}_{A/R,(p-1)_{\sigma}}
\]

\begin{proof}
If an element of $\mathrm P$ is sent to $0 \in \widehat {\mathrm P}^{(0)}$, then it is also sent to $0$ in $\mathrm P^{(0)}_{(0)_{\sigma}} = A$ and it therefore belongs to $I$.
Now an element of the form $\tilde z - z \in I$ is sent to
\[
\widehat \Theta^{(0)}(z) -z = \sum_{k=1}^\infty \partial_{\sigma}^{k}(z) \xi^{[k]} \in \widehat {\mathrm P}^{(0)}
\]
and this is equal to $0$ if and only if $\partial_{\sigma}(z) = 0$.
Thus we see that the kernel of $\mathrm P \to \widehat {\mathrm P}^{(0)}$ is the ideal $J$ generated by the $\tilde z - z$ with $z \in A^{\partial_{\sigma} = 0}$ and we have $\mathrm P/J = A \otimes_{A^{\partial_{\sigma}= 0}} A$.

When $q\mathrm{-char}(R) = p > 0$, the image of $\xi^{(p)}$ in $A\langle \xi \rangle$ is $(p)_{q}!\xi^{[p]} = 0$.
Therefore, there exists an epi-mono factorization
\[
A[\xi] \twoheadrightarrow A[\xi]/\xi^{(p)} \to A \langle \xi \rangle
\]
inducing an isomorphism of $A$-modules
\[
A[\xi]/\xi^{(p)} \simeq A \langle \xi \rangle/I^{[p]}
\]
because $R$ is $q$-divisible.
The second assertion follows immediately. \end{proof}

\begin{rmk}
If $R$ were not $q$-divisible (but still $q\mathrm{-char}(R) = p > 0$), we would still get a decomposition
\[
\widehat {\mathrm P}^{(\infty)}_{A/R,\sigma} \twoheadrightarrow \mathrm P^{(\infty)}_{A/R,(p-1)_{\sigma}} \to \widehat {\mathrm P}^{(0)}_{A/R,\sigma} \twoheadrightarrow \mathrm P^{(0)}_{A/R,(p-1)_{\sigma}} 
\]
and an inclusion $A \otimes_{A^{\partial_{\sigma}= 0}} A \hookrightarrow \mathrm P^{(0)}_{A/R,(p-1)_{\sigma}}$.
\end{rmk}

\section{Twisted differential operators of level zero}

We assume again that $A$ is a twisted commutative algebra with twisted coordinate $x$ and we set $y := x -\sigma(x)$.

If $M$ is an $A$-module, when we write $\mathrm P_{(n)_{\sigma}}^{(0)} \otimes'_{A} M$, we mean that we endow $\mathrm P_{(n)_{\sigma}}^{(0)}$ with the action given by the \emph{twisted Taylor map}.
In other words, we have
\[
\forall z \in A, s \in M, \quad \xi^{[k]} \otimes' zs = \Theta(z)\xi^{[k]} \otimes' s.
\]
In particular, on $\mathrm P_{(n)_{\sigma}}^{(0)} \otimes'_{A} \mathrm P_{(m)_{\sigma}}^{(0)}$, we use the natural action of $A$ for the left structure (action on the right) and the twisted Taylor map for the right structure (action on the left).
Also, it will be convenient to set
\[
\widehat {\mathrm P}_{\sigma}^{(0)} \widehat \otimes'_{A} \widehat {\mathrm P}_{\sigma}^{(0)} := \varprojlim \left(\mathrm P_{(n)_{\sigma}}^{(0)} \otimes'_{A} \mathrm P_{(m)_{\sigma}}^{(0)}\right).
\]
This is the set of infinite sums $\sum_{i,j \in \mathbb N} z_{i,j} \xi^{[i]} \otimes' \xi^{[j]}$ with $z_{i,j} \in A$ (with Taylor switch on coefficients).
\begin{dfn}
If $M$ and $N$ are two $A$-modules, a \emph{twisted differential operator of level $0$} of from $M$ to $N$ is an $A$-linear map
\[
\mathrm P_{(n)_{\sigma}}^{(0)} \otimes'_{A} M \to N.
\]
\end{dfn}

In general, we will write
\[
\mathrm{Diff}_{n,\sigma}^{(0)}(M, N) = \mathrm{Hom} _{A} (\mathrm P_{(n)_{\sigma}}^{(0)} \otimes'_{A} M, N)
\]
and
\[
\mathrm{Diff}_{\sigma}^{(0)}(M, N) = \varinjlim \mathrm{Diff}_{n}^{(0)}(M, N).
\]
In the case $N = M$, we will simply write $\mathrm{Diff}_{n,\sigma}^{(0)}(M)$ and $\mathrm{Diff}_{\sigma}^{(0)}(M)$.
Moreover, we set $\mathrm D_{A/R,\sigma}^{(0)} := \mathrm{Diff}_{\sigma}^{(0)}(A)$.

\begin{dfn} \label{commu}
The \emph{comultiplication on $\mathrm P_{\sigma}^{(0)}$} is the $A$-linear map
\[
\xymatrix@R0cm{\widehat \delta^{(0)} : & \widehat {\mathrm P}_{\sigma}^{(0)} \ar[r] & \widehat {\mathrm P}_{\sigma}^{(0)} \widehat \otimes'_{A} \widehat {\mathrm P}_{\sigma}^{(0)}
\\ & \xi^{[n]} \ar@{|->}[r] & \sum_{i=0}^n \xi^{[n-i]} \otimes' \xi^{[i]} }
\]
\end{dfn}

We might also have to consider the partial comultiplication maps
\[
\delta_{m,n}^{(0)} : \mathrm P_{(m+n)_{\sigma}}^{(0)} \to \mathrm P_{(n)_{\sigma}}^{(0)}  \otimes'_{A} \mathrm P_{(m)_{\sigma}}^{(0)}.
\]
which are given by the same formulas.
In practice, we should simply denote all these maps by $\delta$.

\begin{prop} \label{chanr}
There exists a commutative diagram
\[
\xymatrix{ A \otimes_{R} A \ar@{=}[r] \ar[d]^{\delta}& \mathrm P \ar[r] \ar[d]^{\delta} & \widehat {\mathrm P}^{(\infty)}_{\sigma}\ar[d]^{\widehat \delta^{(\infty)}} \ar[r] & \widehat {\mathrm P}_{\sigma}^{(0)} \ar[d]^{\widehat\delta^{(0)}}
\\ A \otimes_{R} A \otimes_{R} A \ar@{=}[r]& \mathrm P \otimes_{A}' \mathrm P \ar[r] & \widehat {\mathrm P}^{(\infty)}_{\sigma} \otimes_{A}' \mathrm P^{(\infty)}_{\sigma} \ar[r] & \widehat {\mathrm P}_{\sigma}^{(0)} \otimes'_{A} \widehat {\mathrm P}_{\sigma}^{(0)}
}
\]
where the first horizontal map sends $z_{1} \otimes z_{2}$ to $z_{1} \otimes 1 \otimes z_{2}$.
\end{prop}

\begin{proof}
Follows from theorem theorem 3.5 of \cite{LeStumQuiros18}.
\end{proof}

\begin{prop}
The comultiplication map $\widehat \delta^{(0)}$ is a homomorphism of rings.
\end{prop}

Of course, the same result holds for the partial comultiplications.

\begin{proof}
First, we may clearly assume that $A = R[x]$ is the polynomial ring in the variable $x$.
We can then reduce to the case $R = \mathbb Z[t,s]$ with $q = t$ and $h = s$ and finally to the case $R = \mathbb Q(t)[s]$.
In other words, we may assume that all $q$-integers are invertible in $R$.
Then the assertion follows from proposition \ref{compinf} and theorem 3.5 of \cite{LeStumQuiros18} since we know that the comultiplication map is a ring morphism on $\widehat {\mathrm P}^{(\infty)}_{\sigma}$.
Actually, this last result itself follows  from the fact that comultiplication is already a ring morphism on $\mathrm P$ (it corresponds to the projection that forgets the middle term).
\end{proof}

\begin{dfn}
The composition of  two twisted differential operators of level $0$, $\Phi : \mathrm P_{(n)_{\sigma}}^{(0)} \otimes_{A} M \to N$ and $\Psi :\mathrm P_{(m)_{\sigma}}^{(0)} \otimes_{A} L \to M$, is the twisted differential operator of level $0$
\begin{equation} \label{facdif}
\Phi \circ \Psi : \xymatrix@R0cm{ \mathrm P_{(m+n)_{\sigma}}^{(0)} \otimes'_{A} L \ar[r]^-{\delta \otimes \mathrm{Id}} & \mathrm P_{(n)_{\sigma}}^{(0)} \otimes'_{A} \mathrm P_{(m)_{\sigma}}^{(0)} \otimes'_A L \ar[r]^-{\mathrm{Id} \otimes'_{A} \Psi} & \mathrm P_{(n)_{\sigma}}^{(0)} \otimes'_{A} M \ar[r]^-{\Phi} & N.}
\end{equation}
\end{dfn}

\begin{prop}
Composition of twisted differential operators of level $0$ is associative.
In particular, it turns $\mathrm D_{A/R,\sigma}^{(0)}$ into a ring.
\end{prop}

\begin{proof}
We can reduce as usual to the case where $R = \mathbb Q(t)[s]$, $q = t$ and $h=s$ and use the analogous result for twisted differential operators of infinite level (see Proposition 4.7 of \cite{LeStumQuiros18}).
\end{proof}

Recall that we introduced in Definition 5.4 of \cite{LeStumQuiros15a} the \emph{twisted Weyl algebra} $\mathrm D_{A/R,\sigma,\partial_\sigma}$ associated to the twisted differential algebra $A$: this is the Ore extension of $A$ by $\sigma$ and $\partial_{\sigma}$ as in proposition 1.4 of \cite{Bourbaki12}).
Concretely, this is the free $A$-module on abstract generators $\partial_{\sigma}^k$ with the commutation rule $\partial_{\sigma} \circ z = \sigma(z)\partial_{\sigma} + \partial_\sigma(z)$.

\begin{prop}
There exists an isomorphism of filtered $R$-algebras $\mathrm D_{A/R,\sigma,\partial_\sigma} \simeq \mathrm D_{A/R,\sigma}^{(0)}$.
\end{prop}

In the future, we will identify these two rings and simply write $\mathrm D_{A/R,\sigma}$.

\begin{proof}
There exists an obvious isomorphism of filtered $A$-modules $\mathrm D_{A/R,\sigma,\partial_\sigma} \simeq \mathrm D_{A/R,\sigma}^{(0)}$ obtained by making the $\partial_{\sigma}^k$'s dual to the  $\xi^{[k]}$'s.
We only need to show that this is a morphism of rings and, as usual, we may assume that all $q$-integers are invertible.
But then, it follows from proposition \ref{compinf} that there exists an isomorphism of filtered rings $\mathrm D_{A/R,\sigma}^{(\infty)} \simeq \mathrm D_{A/R,\sigma}^{(0)}$.
On the other hand, there exists also a canonical isomorphism of filtered rings $\mathrm D_{A/R,\sigma,\partial_\sigma} \simeq \mathrm D_{A/R,\sigma}^{(\infty)}$ as we saw in theorem 6.3 of \cite{LeStumQuiros18}.
Our isomorphism is obtained by composing them.
\end{proof}

\begin{rmk}
\begin{enumerate}
\item
This last result might give the feeling that we have been working quite hard for nothing:
defining twisted divided powers required some energy.
But this is not true.
The dual approach to the twisted Weyl $R$-algebra introduces new tools that will prove to be quite profitable.
Recall that it is also possible to define twisted differential operators of infinite level inductively as operators on the ring of functions and avoid the introduction of principal parts (and twisted powers).
Again, this might sound simpler but it is not the best way to do it.
\item
The canonical map $A[\xi] \to A\langle \xi \rangle$ is essentially dual to the canonical map $\mathrm D_{A/R,\sigma} \to \mathrm D_{A/R,\sigma}^{(\infty)}$ whose image is the subring $\overline {\mathrm D}_{A/R,\sigma}$ of \emph{small twisted differential operators} generated by functions and derivations inside $\mathrm{End}_{R}(A)$.
\end{enumerate}
\end{rmk}

\begin{prop}
Assume $R$ is $q$-divisible and $q\mathrm{-char}(R) = p > 0$.
Then, there exists a commutative diagram
\[
\xymatrix{
& \mathrm{End}_{A^{\partial_{\sigma}=0}}(A) \ar@{^{(}->}[r] & \mathrm{End}_{R}(A) \\
\mathrm D_{A/R,\sigma} \ar@{->>}[r] & \overline {\mathrm D}_{A/R,\sigma} \ar@{^{(}->}[r] \ar@{^{(}->}[u] & \mathrm D^{(\infty)}_{A/R,\sigma} \ar@{^{(}->}[u] \\
\mathrm{Diff}^{(0)}_{p-1,\sigma}(A) \ar@{=}[rr] \ar@{^{(}->}[u] \ar[ur]^{\simeq} && \mathrm{Diff}^{(\infty)}_{p-1,\sigma}(A). \ar@{^{(}->}[u] \ar[ul]^{\simeq}
}
\]
\end{prop}

\begin{proof}
This is obtained by duality from proposition \ref{epimono}.\end{proof}

\section{Twisted $p$-curvature}

As before, $A$ denotes a twisted $R$-algebra with coordinate $x$.
In particular, we have $\sigma(x) = qx+h$ with $q,h \in R$ and we set $y := x - \sigma(x)$.
We also assume in this section that $q\mathrm{-char}(R) = p > 0$.

\begin{lem} \label{compcan}
For all $n \in \mathbb N$, the diagram
\[
\xymatrix{A \ar@<2pt>[rr]^-{\mathrm{can}} \ar@<-2pt>[rr]_-{\Theta} &&  \mathrm P_{(n)_{\sigma}}^{(0)}\ar[rr] &&  \mathrm P_{(n)_{\sigma}}^{(0)}/ (\xi)}
\]
is commutative.
\end{lem}

It means that, modulo $\xi$, both $A$-algebra structures coincide on $\mathrm P_{(n)_{\sigma}}^{(0)}$.

\begin{proof}
If $I$ denotes the ideal of the diagonal in $\mathrm P := A \otimes_{R} A$ as usual, we may consider the following commutative diagram
\[
\xymatrix{
A \ar@<2pt>[rr]^-{\mathrm{can}} \ar@{=}[d] \ar@<-2pt>[rr]_-{\Theta} && \mathrm P \ar[rr] \ar[d] &&  \mathrm P/I = A \ar[d]
\\
A \ar@<2pt>[rr]^-{\mathrm{can}} \ar@<-2pt>[rr]_-{\Theta} &&  \mathrm P_{(n)_{\sigma}}^{(0)}\ar[rr] &&  \mathrm P_{(n)_{\sigma}}^{(0)}/\xi
}
\]
The upper left maps are given by left and right actions of $A$ on $\mathrm P$ and it follows that the upper line is commutative.
And all the squares are commutative.
Therefore, the second line must be commutative too.
\end{proof}

\begin{prop} \label{comdelt}
For all $m, n \in \mathbb N$, the following diagram is commutative:
\[
\xymatrix{\mathrm P_{((n+m)p)_{\sigma}}^{(0)} \ar[r]^-\delta \ar@{->>}[d] & \mathrm P_{(np)_{\sigma}}^{(0)} \widehat \otimes'_{A} \widehat {\mathrm P}_{(mp)_{\sigma}}^{(0)} \ar@{->>}[d]\\
\mathrm P_{((n+m)p)_{\sigma}}^{(0)}/ (\xi) \ar[r] & \mathrm P_{(np)_{\sigma}}^{(0)}/ (\xi)  \otimes_{A}  \mathrm P_{(mp)_{\sigma}}^{(0)}/ (\xi) \\
A\langle\omega\rangle/I_{\omega}^{[n+m+1]} \ar[r]^-\delta \ar[u] & A\langle\omega\rangle/I_{\omega}^{[n+1]} \otimes_{A} A\langle\omega\rangle/I_{\omega}^{[m+1]}. \ar[u]}
\]
\end{prop}

The upper map $\delta$ comes from definition \ref{commu} and the bottom one is the comultiplication map that we met in proposition \ref{dualcom}.
The bottom vertical maps are induced by the twisted divided $p$-power map \eqref{callu}.

\begin{proof}
By definition, all horizontal arrows are given by compatible formulas on the generators.
However, the upper right tensor product is obtained by using the Taylor map on the right factor although the down right tensor product uses the canonical structure on both side.
But this does not matter thanks to lemma \ref{compcan}.
\end{proof}

\begin{prop} \label{duacent}
Assume that $R$ is $q$-divisible.
Then, there exists a (unique) $A$-linear homomorphism of $R$-algebras
\begin{equation} \label{pcurv}
\xymatrix@R0cm{A[\theta]  \ar[r] & \mathrm{D}_{A/R,\sigma}
\\ \theta \ar@{|->}[r] & \partial_{\sigma}^{p}}
\end{equation}
It induces an isomorphism between $A[\theta]$ and the centralizer $A\mathrm{Z}_{A/R,\sigma}$ of $A$ in $\mathrm{D}_{A/R,\sigma}$ and an isomorphism between $A^{\partial_{\sigma}=0}[\theta]$ and the center $\mathrm{Z}_{A/R,\sigma}$ of $\mathrm{D}_{A/R,\sigma}$.
\end{prop}

\begin{proof}
We know from the first part of proposition \ref{dualcom} that the bottom map of proposition \ref{comdelt} is dual to multiplication on the polynomial ring $A[\theta]$.
And by definition, the top map is dual to multiplication on $\mathrm{D}_{A/R,\sigma}$.
Moreover, since we assume that $R$ is $q$-divisible, it follows from theorem \ref{dualpc} that the bottom vertical maps of proposition \ref{comdelt} are bijective.
Therefore, by duality, the top vertical maps corresponds to an injective morphism of $R$-algebras $A[\theta] \to \mathrm{D}_{A/R,\sigma}$ that sends $\theta$ to $\partial_{\sigma}^p$.
Since $A[\theta]$ is a commutative ring, its image is contained into the centralizer $A\mathrm{Z}_{A/R,\sigma}$ of $A$ in $\mathrm{D}_{A/R,\sigma}$.
Conversely, since $R$ is $q$-flat, it follows from the first part of lemma \ref{centcomp} below that the image of $A[\theta]$ is exactly $A\mathrm{Z}_{A/R,\sigma}$.
The assertion about $\mathrm{Z}_{A/R,\sigma}$ is then a consequence of the last assertion of the same lemma.
\end{proof}
 
\begin{lem} \label{centcomp}
We denote by $A[\partial_{\sigma}^p]$ (resp. $A^{\partial_{\sigma}=0}[\partial_{\sigma}^p]$) the $A$-submodule (resp. $A^{\partial_{\sigma}=0}$-module) of $\mathrm D_{A/R,\sigma}$ generated by $\partial_{\sigma}^{pk}$ with $k \in \mathbb N$.
Then,
\begin{enumerate}
\item if $A$ is $q$-flat, we have $A\mathrm {Z}_{A/R,\sigma} \subset A[\partial_{\sigma}^p]$,
\item we always have $A[\partial_{\sigma}^p] \cap \mathrm{Z}_{A/R} = A^{\partial_{\sigma}=0}[\partial_{\sigma}^p] \cap A\mathrm {Z}_{A/R,\sigma}$.
\end{enumerate}
\end{lem}

Be careful that, in this lemma, $A[\partial_{\sigma}^p]$ and $A^{\partial_{\sigma}=0}[\partial_{\sigma}^p]$ denote the $A$-submodules generated by the powers of $\partial_{\sigma}^p$ which are \emph{a priori} different from the $R$-subalgebra generated by $A$ and $\partial_{\sigma}^p$ (as long as this last ring is not known to be commutative for example).

\begin{proof}
If $\varphi := \sum z_{k} \partial_{\sigma}^k \in \mathrm D_{A/R,\sigma}$, we can use proposition 6.4 from \cite{LeStumQuiros18} and write
\[
\varphi x = \sum z_{k} \partial_{\sigma}^kx = \sum z_{k}(\sigma^k(x) \partial_{\sigma}^{k} + (k)_{q}\partial_{\sigma}^{k-1}) 
\]
\[
= \sum \sigma^k(x)z_{k} \partial_{\sigma}^{k} + \sum (k)_{q}z_{k}\partial_{\sigma}^{k-1} = \sum  \left(\sigma^k(x)z_{k} + (k+1)_{q}z_{k+1}\right) \partial_{\sigma}^{k}.
\]
Therefore, if $\varphi$ commutes with $x$, we will have
\[
\forall k \geq 0, \quad  \sigma^k(x)z_{k} + (k+1)_{q}z_{k+1} = xz_{k}.
\]
For $k = 0$, we obtain $z_{1} = 0$.
If $k$ is a positive integer such that $z_{k-1} = 0$, we must have $(k)_{q}z_{k} = 0$.
If we assume that $A$ is $q$-flat, we must have $(k)_{q} = 0$ or $z_{k} = 0$.
Since $q\mathrm{-char}(R) = p >0$, this exactly means that $\varphi \in A[\partial_{\sigma}^{p}]$.

We now prove the second assertion.
We pick-up some
\[
\varphi := \sum z_{k} \partial_{\sigma}^{kp} \in A[\partial_{\sigma}^p]\cap A\mathrm {Z}_{A/R,\sigma}.
\]
Then, we have $\varphi \in \mathrm{Z}_{A/R}$ if and only if $\varphi \partial_{\sigma} = \partial_{\sigma} \varphi$ which means that
\[
\sum z_{k} \partial_{\sigma}^{kp+1} = \sum  \partial_{\sigma} z_{k} \partial_{\sigma}^{kp} =  \sum  \sigma(z_{k}) \partial_{\sigma}^{kp+1} +  \partial_{\sigma}(z_{k}) \partial_{\sigma}^{kp}.
\]
Thus we see that the condition is equivalent to
\[
\forall k \in \mathbb N, \quad z_{k} =  \sigma(z_{k}) \quad \mathrm{and} \quad \partial_{\sigma}(z_{k}) = 0.
\]
It follows from lemma 6.4 of \cite{LeStumQuiros15a} for example that the first condition is implied by the second and we are done.
\end{proof}

\begin{dfn}
The map \eqref{pcurv} is the \emph{twisted $p$-curvature map}.
\end{dfn}

\section{Divided Frobenius} \label{secsix}

In this section, the ring $R$ is endowed with an endomorphism $F^*_{R}$, $A$ denotes a commutative $R$-algebra and $x$ is any element of $A$.
We set $y := (1-q)x$.
We also fix a $p \in \mathbb N \setminus \{0\}$ and at some point, we will use $q' := q^p$  and $y' := (p)_{q}y$.

Recall also that we write for all $n \in \mathbb N$,
\[
\xi^{(n)} := \xi^{(n)_{q,y}}:= \prod_{i=0}^{n-1}\left(\xi + (i)_{q}y\right) = \prod_{i=0}^{n-1}\left(\xi + (1 -q^i)x\right) = \prod_{i=0}^{n-1}\left(x + \xi -q^ix\right).
\]

\begin{dfn}
A \emph{$p$-Frobenius} on $A$ (with respect to $F_{R}^*$ and $x$) is a morphism of $R$-algebras $F_{A/R}^* : A' := R {}_{{}_{F^*_{R}}\nwarrow}\!\!\otimes_{R} A \to A$ such that $F_{A/R}^*(1 \otimes x) = x^p$.
\end{dfn}

\begin{xmp}
\begin{enumerate}
\item
If $R$ is a ring of prime characteristic $p >0$ endowed with the $p$th power map, then the usual relative Frobenius is a $p$-Frobenius on $A$.
\item
If $R$ is a ring of $p^N$-torsion with $p$ prime, $F^*$ is a lifting of the $p$th power map on $R/p$ and $x$ is an \'etale coordinate on $A$, then there exists a unique $p$-Frobenius on $A$.
\item
If $A = R[x]$ or $A = R[x, x^{-1}]$, then there exists a unique $p$-Frobenius on $A$.
\end{enumerate}
\end{xmp}

\begin{dfn} \label{frobd}
If $F_{A/R}^*$ is a  $p$-Frobenius on $A$, then the \emph{$p$-Frobenius} $F^*_{A[\xi]/R}$ on $A[\xi]$ is the $F_{A/R}^*$-linear morphism of $R$-algebras
\[
\xymatrix@R=0cm{
F_{A/R}^* : A'[\xi] \ar[r] & A[\xi] \\ \xi \ar@{|->}[r] & (x + \xi)^p - x^p.
}
\]
\end{dfn}

\begin{rmk}
\begin{enumerate}
\item
The $p$-Frobenius on $A[\xi]$ is both a $p$-Frobenius with respect to $x$ and to $x + \xi$.
\item
When $R$ has $q$-characteristic $p$, then $(x + \xi)^p - x^p = \xi^{(p)}$ (use proposition 4.6 of \cite{LeStumQuiros15} for example).
\item \label{rem3}
There exists a commutative diagram
\[
\xymatrix{A'[\xi] \ar[rr]^{F^*} \ar[d] &&A[\xi] \ar[d] \\ A' \otimes_{R} A' \ar[rr]^{F^* \otimes F^*} &&A \otimes_{R} A
}
\]
where the vertical map sends $\xi$ to $1 \otimes x - x \otimes 1$.
\end{enumerate}
\end{rmk}

We will frequently need the twisted binomial formula (see proposition 2.14 of \cite{LeStumQuiros15} for example) that we recall now:
\begin{equation} \label{bino1}
\forall z_{1}, z_{2} \in A, \forall n \in \mathbb N, \quad \prod_{i=0}^{n-1}(q^iz_{1}+z_{2}) = \sum_{k=0}^n q^{\frac {k(k-1)}2} {n \choose k}_q z_{1}^{k}z_{2}^{n-k}.
\end{equation}

Now, we become also interested in the Frobenius version of the twisted powers.
Recall that we write $q' := q^p$ and $y' := (p)_{q}y$ and we have therefore
\[
\forall n \in \mathbb N, \quad \xi^{(n)_{q',y'}} := \prod_{i=0}^{n-1}\left(\xi + (i)_{q^p}(p)_{q}y\right) =\prod_{i=0}^{n-1}\left(\xi + (pi)_{q}y\right) = \prod_{i=0}^{n-1}\left(x + \xi -q^{pi}x\right).
\]

\begin{dfn}
The \emph{$p$-Frobenius coefficients} are the polynomials
\[
A_{n,i} := \sum_{j=0}^n (-1)^{n-j}  t^{\frac {p(n-j)(n-j-1)}2}   {n \choose j}_{t^p}  {pj \choose i}_{t} \in \mathbb Z[t].
\]
\end{dfn}

\begin{rmk}
We will show later that $A_{n,i} = 0$ unless $n \leq i \leq pn$ but we may observe right now that $A_{n,pn} =  1$ and that $A_{n,i} = 0$ for $i > pn$.
\end{rmk}

From now on, we will often omit the index in the $p$-Frobenius maps and simply write $F^*$.

\begin{prop}
If $F^*$ is a  $p$-Frobenius on $A$, then we have
\[
\forall n \in \mathbb N, \quad F^*(\xi^{(n)_{q',y'}}) = \sum_{i=0}^{pn} A_{n,i}(q) x^{pn-i}\xi^{(i)_{q,y}}
\]
where the $A_{n,i}$ are the $p$-Frobenius coefficients.
\end{prop}

\begin{proof}
Using the twisted binomial formula \eqref{bino1} in the case $z_{1} = x^p$ and $z_{2} = -(\xi + x)^p$ (with $q^p$ instead of $q$), we have
\begin{eqnarray*}
F^*(\xi^{(n)_{q',y'}}) & =& \prod_{i=0}^{n-1}\left((x+ \xi)^p-q^{pi}x^p\right) = 
\\
& = &\sum_{j=0}^n (-1)^{n-j} q^{\frac {p(n-j)(n-j-1)}2} {n \choose j}_{q^p}   x^{p(n-j)}(x + \xi)^{pj}.
\end{eqnarray*}
Using lemma \ref{lembin} below, we obtain
\begin{eqnarray*}
F^*(\xi^{(n)_{q',y'}}) &=& \sum_{j=0}^n (-1)^{n-j} q^{\frac {p(n-j)(n-j-1)}2} {n \choose j}_{q^p}  x^{p(n-j)} \sum_{i=0}^{pj} {pj \choose i}_{q} x^{pj-i}\xi^{(i)_{q,y}}
\\
&=& \sum_{i=0}^{pn} \left(\sum_{j=0}^n (-1)^{n-j}  q^{\frac {p(n-j)(n-j-1)}2}   {n \choose j}_{q^p}  {pj \choose i}_{q} \right) x^{pn-i}\xi^{(i)_{q,y}}.\quad \Box
\end{eqnarray*}\end{proof}

\begin{lem} \label{lembin}
We have for all $m \in \mathbb N$,
\begin{equation} \label{binalg}
(x + \xi)^{m} = \sum_{i=0}^{m} {m \choose i}_{q} x^{m-i}\xi^{(i)}.
\end{equation}
\end{lem}

\begin{proof}
We have for all $i \in \mathbb N$, $\xi^{(i+1)} = \xi^{(i)}(x + \xi -q^ix)$ and it follows that $\xi^{(i)}(x + \xi) = \xi^{(i+1)} + q^{i}x\xi^{(i)}$.
Therefore, if the formula holds for some $m$, we will have
 \begin{eqnarray*}
 (x + \xi)^{m+1} &=& \sum_{i=0}^{m} {m \choose i}_{q} x^{m-i}\xi^{(i)}(x + \xi)
\\
 &=& \sum_{i=0}^{m} {m \choose i}_{q} x^{m-i}(\xi^{(i+1)} + q^ix\xi^{(i)})
\\
&=& \sum_{i=1}^{m+1} {m \choose i-1}_{q} x^{m-i+1}\xi^{(i)} + \sum_{i=0}^{m} {m \choose i}_{q} q^ix^{m+1-i}\xi^{(i)}
\\
&=& \sum_{i=0}^{m+1} \left({m \choose i-1}_{q} + q^i{m \choose i}_{q} \right) x^{m+1-i}\xi^{(i)}
\\
&= &\sum_{i=0}^{m+1} {m+1 \choose i}_{q} x^{m-i}\xi^{(i)}.\qedhere
\end{eqnarray*}\end{proof}

As a particular case of the proposition, we have the following:

\begin{cor}
If $F^*$ is a $p$-Frobenius on $A$, then we have
\[
F^*(\xi) = \sum_{i=1}^{p} {p \choose i}_{q} x^{p-i}\xi^{(i)}. \quad  \Box
\]
\end{cor}

As a preparation for the next statement, we prove now the following exchange lemma:

\begin{lem} \label{techf}
We have for all $m, n \in \mathbb N$,
\[
q^{\frac{n(n-1)}{2}} (1-q)^n (n)_{q}! { m \choose n}_{q}= \sum_{k=0}^n (-1)^{n-k}q^{\frac{k(k-1)}{2}}{n \choose k}_{q}q^{m(n-k)}.
\]
\end{lem}

It means in particular that the right hand side is zero unless $m \geq n$.

\begin{proof}
Using the twisted binomial formula \eqref{bino1} for $z_{1}= 1$ and $z_{2}=-q^{m}$, we get
\begin{eqnarray*}
\sum_{k=0}^n (-1)^{n-k}q^{\frac{k(k-1)}{2}}{n \choose k}_{q}q^{m(n-k)} & = & \prod_{k=0}^{n-1}(q^{k}-q^{m})
\\
&=& \prod_{k=0}^{n-1}q^{k}\prod_{k=0}^{n-1}(1 - q^{m-k})
\\
&=& q^{\frac{n(n-1)}{2}} (1 - q)^n \prod_{k=0}^{n-1}(m-k)_{q}
\\
&=&q^{\frac{n(n-1)}{2}} (1-q)^n (n)_{q}! { m \choose n}_{q}.\qedhere
\end{eqnarray*}\end{proof}

\begin{lem}\label{techf2}
Given $n, i \in \mathbb N$, we have $A_{n,i} = 0$ unless $n \leq i \leq pn$ in which case
\[
q^{\frac{i(i-1)}{2}}(1-q)^{i-n}(i)_{q}!A_{n,i}(q) = (p)_q^n (n)_{q^p}!q^{p\frac{n(n-1)}{2}}\sum_{l=0}^{i-n} (-1)^{i-n+l} q^{\frac {l(l-1)}2} {i \choose l}_q {i-l \choose n}_{q^p}.
\]
\end{lem}

\begin{proof}
We will compute
\[
LHS := q^{\frac{i(i-1)}{2}}(1-q)^i(i)_{q}!A_{n,i}(q).
\]
In order to do that, we use lemma \ref{techf} twice (with $q^p$ instead of $q$ the second time):
\begin{eqnarray*}
LHS &=& q^{\frac{i(i-1)}{2}}(1-q)^i(i)_{q}!\sum_{k=0}^n (-1)^{n-k}  q^{\frac {p(n-k)(n-k-1)}2}   {n \choose k}_{q^p}  {pk \choose i}_{q}
\\
&=& \sum_{k=0}^n (-1)^{n-k}  q^{\frac {p(n-k)(n-k-1)}2}   {n \choose k}_{q^p} \left(q^{\frac{i(i-1)}{2}}(1-q)^i(i)_{q}! {pk \choose i}_{q}\right)
\\
&=&\sum_{k=0}^n(-1)^{n-k}q^{p\frac{(n-k)(n-k-1)}{2}}{n \choose k}_{q^p}\left( \sum_{l=0}^i (-1)^{i-l}{i \choose l}_q q^{\frac {l(l-1)}2}q^{pk(i-l)}\right)
\\
&=&\sum_{l=0}^i (-1)^{i-n+l}{i \choose l}_q q^{\frac {l(l-1)}2} \left(\sum_{k=0}^n(-1)^{k}q^{p\frac{(n-k)(n-k-1)}{2}}{n \choose k}_{q^p}q^{pk(i-l)}\right).
\\
&=& \sum_{l=0}^{i-n}(-1)^{i-n+l}{i \choose l}_q q^{\frac {l(l-1)}2} \left(q^{p\frac{n(n-1)}{2}} (1-q^p)^n (n)_{q^p}! {i-l \choose n}_{q^p}\right)
\\
&=& (p)_q^n (n)_{q^p}!q^{p\frac{n(n-1)}{2}}(1-q)^n\sum_{l=0}^{i-n} (-1)^{i-n+l} q^{\frac {l(l-1)}2} {i \choose l}_q {i-l \choose n}_{q^p}
\end{eqnarray*}
since $1 - q^p = (1-q)(p)_{q}$.
When $i < n$, the right hand side is zero.
Since $A_{n,i}$ is a polynomial in $q$, it has to be zero too.
Otherwise, we obtain the expected equality by moving $(1-q)^n$ to the left hand side.
\end{proof}

\begin{rmk}
In particular, we have
\[
(n)_{q}!A_{n,n} (q)  = (p)_{q}^n(n)_{q^p}!q^{(p-1)\frac{n(n-1)}{2}}.
\]
\end{rmk}

\begin{prop}\label{adolfo}
Given $n, i \in \mathbb N$, there exists a unique $B_{n,i} \in \mathbb Z[t]$ such that
\[
(i)_{t}!A_{n,i} = (n)_{t^p}!(p)_{t}^n B_{n,i}(t),
\]
where $A_{n,i}$ denote the $p$-Frobenius coefficient.
We have $B_{n,i}  = 0$ unless $n \leq i \leq pn$ with extreme values
\[
B_{n,n}(q) = q^{\frac{(p-1)n(n-1)}{2}} \quad \mathrm{and} \quad B_{n,pn}(q) = \prod_{k=1}^n
\prod_{i=1}^{p-1} (kp-i)_{q}.
\]
\end{prop}

\begin{proof}
Any non zero $t$-integer or $t^p$-integer is prime to both $1-t$ and $t$.
The first assertion therefore follows from lemma \ref{techf2}.
The precise values in the case $i = n$ and $i = pn$ are obtained from the remark before the proposition and from the fact that $A_{n,pn} = 1$ since
\begin{eqnarray*}
(pn)_{q}! &=& \prod_{k=1}^n \prod_{i=0}^{p-1} (kp-i)_{q}
\\
&=& \prod_{k=1}^n (kp)_{q} \prod_{k=1}^n \prod_{i=1}^{p-1} (kp-i)_{q}
\\
&=& (p)_{q}^n(n)_{q^p}! \prod_{k=1}^n \prod_{i=1}^{p-1} (kp-i)_{q}
\end{eqnarray*}
because $(kp)_{q} = (k)_{q^p}(p)_{q}$ for each $k$.
\end{proof}

\begin{xmp}
\begin{enumerate}
\item
We have $B_{1,1}(q) = 1$, $B_{1,2}(q) = (p-1)_{q}$, $B_{2,2}(q) = q^{p-1}$, $B_{3,3}(q) = q^{3(p-1)}$.
\item When $R$ has positive $q$-characteristic $p$ and $1 \leq n \leq p$, we have
\[
(1-q)^{p-n}B_{n,p}(q) = (-1)^{n-1} {p \choose n}.
\]
For example, if we write $j = \frac {1 + \sqrt{-3}}2$, we obtain $ (1-j)B_{2,3}(j) = - 3$ and therefore $B_{2,3}(j) = j^2 -1$.
\end{enumerate}
\end{xmp}

\begin{dfn}
Let $F^*$ be a $p$-Frobenius on $A$.
Then,
\begin{enumerate}
\item
the \emph{divided $p$-Frobenius coefficients} are the polynomials $B_{n,i}$ of proposition \ref{adolfo},
\item
the \emph{divided $p$-Frobenius map} is the unique $F^*$-linear map $A'\langle\omega\rangle_{q^p,y} \to A\langle\xi\rangle_{q,y}$ such that
\[
\forall n \in \mathbb N, \quad [F^*](\omega^{[n]}) = \sum_{i=n}^{pn} B_{n,i}(q) x^{pn-i}\xi^{[i]}.
\]
\end{enumerate}
\end{dfn}

\begin{rmk}
\begin{enumerate}
\item
As a particular case of this definition, we have
\[
[F^*](\omega) = \sum_{i=1}^{p} {(p-1)_{q}}\cdots (p-i+1)_{q}x^{p-i}\xi^{[i]}.
\]
In more fancy terms, the $i$th coefficient is $(i-1)_{q}!{p-1 \choose i -1}_{q}$.
\item The divided $p$-Frobenius map is continuous.
More precisely, it is compatible with the ideal filtration and induces $F^*$-linear maps
\[
A'\langle \omega \rangle/I_{\omega}^{[n+1]} \to A\langle \xi \rangle/I_{\xi}^{[n+1]}.
\]
\item We may extend the divided Frobenius map by linearity and obtain an $A$-linear map
\[
A\langle\omega\rangle_{q^p,(1-q)x^p} \to A\langle\xi\rangle_{q,y}
\]
given by the same formula (we have $F^{*}(y) = (1-q)x^p$).
\end{enumerate}
\end{rmk}
\begin{lem} \label{subs}
Let $F^*$ be a $p$-Frobenius on $A$.
Then, under the canonical map $A[\xi] \to A\langle \xi \rangle_{q,y}$, we have for all $n \in \mathbb N$,
\[
F^*(\xi^{(n)_{q',y'}}) \mapsto (n)_{q^p}!(p)_{q}^n  [F^*](\omega^{[n]_{q^p,y}}).
\]
\end{lem}

\begin{proof}
This is a direct consequence of the definitions.
More precisely, since $A_{n,i} = 0$ for $i < n$, we have $F^*(\xi^{(n)_{q',y'}}) = \sum_{i=n}^{pn} A_{n,i}(q) x^{pn-i}\xi^{(i)_{q,y}}$ and this is sent to
\begin{eqnarray*}
\sum_{i=n}^{pn} A_{n,i}(q) x^{pn-i} (i)_{q}!\xi^{[i]_{q,y}} &=& \sum_{i=n}^{pn} (n)_{q^p}!(p)_{q}^n B_{n,i}(q) x^{pn-i}\xi^{[i]_{q,y}} \\ &=& (n)_{q^p}!(p)_{q}^n  [F^*](\omega^{[n]_{q^p,y}}).\qedhere
\end{eqnarray*}\end{proof}

\begin{prop} \label{gooddf}
If $F^*$ is a $p$-Frobenius on $A$, then the divided Frobenius map
\[[F^*] : A'\langle\omega\rangle_{q^p,y} \to A\langle\xi\rangle_{q,y}
\]
is a homomorphism of rings.
\end{prop}

\begin{proof}
We want to check that for all $m, n \in \mathbb N$, we have
\begin{equation} \label{morp}
[F^*](\omega^{[m]_{q^p,y}}\omega^{[n]_{q^p,y}}) = [ F]^*(\omega^{[m]_{q^p,y}}) [F^*](\omega^{[n]_{q^p,y}}).
\end{equation}
Note that it is sufficient to do the case $R = \mathbb Z[t]$, $t = q$ and $A = R[x]$, and then specialize our variables.
In particular, we may assume that $q\mathrm{-char}(R) = 0$ in which case we will identify $A[\xi]$ with $A\langle \xi \rangle_{q,y}$.
Then, this essentially follows from the fact that $F^*$ itself is a ring homomorphism.
But we need to be careful.
By $F^*$-linearity, the left hand side of \eqref{morp} is equal to
\[
LHS = \sum_{i=0}^{\min{(m,n)}} (-1)^iq^{\frac{pi(i-1)}2}{m + n -i \choose m}_{q^p}{m \choose i}_{q^p} (1-q)^ix^{pi}[F^*](\omega^{[m+n-i]_{q^p,y}})
\]
From proposition \ref{subs}, we see that, for all $i \leq m +n$, we have
\[
(i)_{q^p}! {m+ n \choose i}_{q^p} (p)_{q}^i F^*(\xi^{(m+n-i)_{q',y'}}) = (m+n)_{q^p}!(p)_{q}^{n+m} [F^*](\omega^{[m+n-i]_{q^p,y}})
\]
and it follows that
\[
(m+ n)_{q^p}!(p)_{q}^{m+n} LHS = 
\]
\[
\sum_{i=0}^{\min{(m,n)}} (-1)^iq^{\frac{pi(i-1)}2} (i)_{q^p}! {m + n -i \choose m}_{q^p}{m \choose i}_{q^p} {m+ n \choose i}_{q^p} (p)_{q}^i (1-q)^ix^{pi} F^*(\xi^{(m+n-i)_{q',y'}})
\]
On the other hand, if we denote the right hand side of \eqref{morp} by by $RHS$, we have

\begin{eqnarray*}
(m+n)_{q^p}!(p)_{q}^{m+n} RHS &=& {m+n \choose n}_{p^q}(m)_{q^p} !(p)_{q}^m [ F]^*(\omega^{[m]_{q^p,y}}) (n)_{q^p}! (p)_{q}^n [ F]^*(\omega^{[n]_{q^p,y}})
\\
&=& {m+n \choose n}_{q^p} F^*(\xi^{(m)_{q',y'}}) F^*(\xi^{(n)_{q',y'}}) 
\\
&=& {m+n \choose n}_{p^q} F^*(\xi^{(m)_{q',y'}} \xi^{(n)_{q',y'}})
\end{eqnarray*}
and finally
\[
(m+n)_{q^p}!(p)_{q}^{m+n} RHS =
\]
\[
{m+n \choose n}_{q^p} \sum_{i=0}^{\min{(m,n)}} (-1)^i (i)_{q^p}! q^{\frac{pi(i-1)}2}{m \choose i}_{q^p}{n \choose i}_{q^p} (p)_{q}^i(1-q)^ix^{pi}F^*(\xi^{(m+n-i)_{q',y'}}).
\]
Now, we may identify both sides because
\[
{m + n -i \choose m}_{q^p} {m+ n \choose i}_{q^p} = {m+n \choose n}_{q^p} {n \choose i}_{q^p}.\qedhere
\]\end{proof}

\begin{prop} \label{frobis}
If $R$ has positive $q$-characteristic $p$ and $F^*$ is a $p$-Frobenius on $A$, then $[F^*]$ induces a homomorphism of $A$-algebras
\[
(A[\xi]/\xi^{(p)_{q,y}}) \langle \omega \rangle_{1, (1-q)x^p} \simeq A[\xi]/\xi^{(p)_{q,y}} \otimes_{A'} A'\langle \omega \rangle_{1, y} \to A \langle \xi \rangle_{q,y}
\]
When $R$ is $q$-divisible, this is an isomorphism.
\end{prop}

\begin{proof}
There exists such a map because $\xi^{(p)}$ is sent to $(p)_{q}\xi^{[p]} = 0$.
Moreover, when $R$ is $q$-divisible, this map induces a bijection between basis on both sides as we can easily check.
More precisely, one may define a notion of degree on both sides by setting $\deg(\xi^{[n]}) = n$ on the right hand side and $\deg (\overline \xi^k \omega^{[n]}) = k+pn$ when $k < p$ on the left hand side.
By definition, this homomorphism preserves the degrees and it is therefore sufficient to prove that it induces a bijection on the associated graded modules.
But then, $\overline \xi^k \omega^{[n]}$ is sent to $B_{n,pn}(q) \xi^{[pn+k]}$ and one has
\[
B_{n,pn}(q) = \left((p-1)_{q}!\right)^n \in R^\times
\]
since $R$ is $q$-divisible.
\end{proof}

\begin{rmk}
\begin{enumerate}
\item This homomorphism is continuous.
Actually, it preserves the ideal filtrations.
Note however that it is \emph{not} an isomorphism of filtered modules when $R$ is $q$-divisible: the filtration on the left hand side is usually strictly smaller that the filtration on the right hand side.
\item
It is tempting to introduce a variant of the $p$-Frobenius coefficients by setting $C_{n,i} = B_{n,i}/B_{n,pn} \in \mathbb Q[t]$.
When $R$ is $q$-divisible, $C_{n,i}(q)$ is well defined and satisfies
\[
(i)_{q}!A_{n,i}(q) = (pn)_{q}! C_{n,i}(q).
\]
Then, if we replace $B$'s with $C$'s in the definition of $[F^*]$, the modified version would send monic to monic (this was our approach in \cite{GrosLeStumQuiros10}).
If moreover, we assume that $q\mathrm{-char}(R) = p$, then the modified version of $[F^*]$ would still be a ring homomorphism (but this is not true anymore in general: this is why we had to be careful in the proof of proposition \ref{gooddf}).
\item There exists an intermediate alternative for the coefficients that is only defined when $R$ is $q$-divisible but which is always a ring homomorphism and coincides with the $C$'s when the $q$-characteristic is $p$.
This is obtained by dividing out $B_{n,i}(q)$ by $((p-1)_{q}!)^n$.
\end{enumerate}
\end{rmk}

\section{Twisted Simpson correspondence}

We let $A$ be a twisted $R$-algebra with twisted coordinate $x$ such that $\sigma(x) = qx$.
We fix an endomorphism $F^*_{R}$ of $R$ and let $F^*$ be a $p$-Frobenius on $A$ with respect to $F^*_{R}$ and $x$ for some $p \in \mathbb N \setminus \{0\}$.
We are mostly interested in the case where $R$ is $q$-divisible of positive $q$-characteristic $p$.

\begin{prop}
Assume that $R$ is $q$-divisible of $q$-characteristic $p$.
Then, the divided $p$-Frobenius provides by an $A$-linear map
\[
\Phi_{A/R} : \mathrm D_{A/R,\sigma} \to \mathrm ZA_{A/R,\sigma} \subset \mathrm D_{A/R,\sigma}.
\]
More precisely, we have for all $n \in \mathbb N$,
\begin{equation} \label{formphi}
\Phi(\partial_{\sigma}^n) = \sum_{k=0}^n B_{k,n}(q) x^{pk-n}\partial_{\sigma}^{pk}
\end{equation}
where the $B_{k,n}$ denote the divided $p$-Frobenius coefficients.
\end{prop}

\begin{proof}
The linearized divided $p$-Frobenius maps $A\langle \omega \rangle/I_{\omega}^{[n+1]} \to A\langle \xi \rangle/I_{\xi}^{[n+1]}$ coming from section \ref{secsix} provide by duality a compatible system of morphisms
\[
\mathrm{Diff}_{n,\sigma}^{(0)}(A) \to A[\theta]_{\leq n}.
\]
We may then use proposition \ref{duacent} in order to identify $A[\theta]$ with the centralizer $\mathrm ZA_{A/R,\sigma}$ of $A$ in $\mathrm D_{A/R,\sigma}$.
By duality, the coefficient of $\partial_{\sigma}^{kp}$ in $\Phi(\partial_{\sigma}^n)$ is the coefficient of $\xi^{[n]}$ in $[F^*](\omega^{[k]})$ which is exactly $B_{k,n} x^{pk-n}$.
\end{proof}

\begin{rmk}
\begin{enumerate}
\item
The morphism $\Phi_{A/R}$ is \emph{not} a ring homomorphism (as we already knew from the case $q = 1$).
\item
If we do not assume that $R$ is $q$-divisible of $q$-characteristic $p$, then we still get a map $\mathrm D_{A/R,\sigma} \to A[\theta]$ given by an analogous formula but we cannot identify the target with $\mathrm ZA_{A/R,\sigma}$.
\item
In formula \eqref{formphi}, the sum actually starts with the smallest integer $k \geq n/p$.
\end{enumerate}
\end{rmk}

\begin{xmp}
\begin{enumerate}
\item
We have $\Phi(\partial_{\sigma}) = x^{p-1}\partial_{\sigma}^p$.
\item
We have $\Phi(\partial_{\sigma}^2) = (p-1)_{q}x^{p-2}\partial_{\sigma}^p + q^{p-1}x^{2p-2}\partial_{\sigma}^{2p}$.
\item When $q\mathrm{-char}(R) = p = 3$, we have
\[
\Phi(\partial_{\sigma}^3) = \partial_{\sigma}^3 + (q^2-1) x^{3}\partial_{\sigma}^{6} + x^{6}\partial_{\sigma}^{9}.
\]
\end{enumerate}
\end{xmp}

Recall from proposition \ref{leseq} that $F^*(A') \subset A^{\partial_{\sigma} = 0}$ if and only the diagram
\[
\xymatrix{A' \ar[r]^{F^*} & A \ar@<2pt>[rr]^-{\mathrm{can}} \ar@<-2pt>[rr]_-{\theta} && \widehat {\mathrm P}_{\sigma}^{(0)}}
\]
commutes.
Under this hypothesis, $[F^*]$ will induce, for all $n \in \mathbb N$, an $F^*$-linear morphism of $R$-algebras
\[
[F^*] : A'\langle \omega \rangle/I_{\omega}^{[n]} \otimes_{A'} A'\langle \omega \rangle/I_{\omega}^{[n]} \to \mathrm P^{(0)}_{(n)_{\sigma}} \otimes'_{A} \mathrm P^{(0)}_{(n)_{\sigma}}.
\]

\begin{prop} \label{comdlef}
If $F^*(A') \subset A^{\partial_{\sigma} = 0}$, then the diagram
\[
\xymatrix{
A'\langle \omega \rangle/I_{\omega}^{[n]} \ar[r]^{[F^*]} \ar[d]^\delta & \mathrm P^{(0)}_{(n)_{\sigma}} \ar[d]^\delta
\\
A'\langle \omega \rangle/I_{\omega}^{[n]} \otimes_{A'} A'\langle \omega \rangle/I_{\omega}^{[n]} \ar[r]^-{[F^*]} & \mathrm P^{(0)}_{(n)_{\sigma}} \otimes'_{A} \mathrm P^{(0)}_{(n)_{\sigma}}.
}
\]
is commutative.
\end{prop}

\begin{proof}
We want to prove that we always have
\[
\delta([F^*](\omega^{[n]}) = [F^*](\delta(\omega^{[n]}).
\]
We can compute the left hand side
\[
\delta([F^*](\omega^{[n]}) = \delta\left(\sum_{k=n}^{pn} B_{n,k}(q) x^{pn-k}\xi^{[k]}\right) = \sum_{k=n}^{pn} B_{n,k}(q) x^{pn-k}\delta(\xi^{[k]})
\]
\[
=  \sum_{k=n}^{pn} B_{n,k}(q) x^{pn-k}\left(\sum_{j=0}^k(\xi^{[k-j]} \otimes' \xi^{[j]}\right)
\]
And we can also compute the right hand side
\[
[F^*](\delta(\omega^{[n]}) = [F^*]\left(\sum_{k=0}^n \omega^{[n-k]} \otimes \omega^{[k]}\right) = \sum_{k=0}^n [F^*](\omega^{[n-k]}) \otimes' [F^*](\omega^{[k]}).
\]
Our assertion therefore follows from lemma \ref{cocomp} below.
\end{proof}

\begin{lem} \label{cocomp}
We have
\[
\sum_{k=n}^{pn} B_{n,k}(q) x^{pn-k}\left(\sum_{j=0}^k \xi^{[k-j]} \otimes' \xi^{[j]}\right) = \sum_{k=0}^n [F^*](\omega^{[n-k]}) \otimes' [F^*](\omega^{[k]})
\]
in $\widehat {\mathrm P}^{(0)}_{A/R} \widehat \otimes'_{A} \widehat {\mathrm P}^{(0)}_{A/R}$.
\end{lem}

\begin{proof}
Since it is a generic question, we may assume that all $q$-integers are invertible in $R$ and also, if we wish, that $A = R[x]$ is simply the polynomial ring.
Using proposition \ref{chanr} and remark \ref{rem3}) after definition \ref{frobd}, it is therefore sufficient to check the equality
\[
\sum_{k=n}^{pn} A_{n,k}(q) x^{pn-k}\left(\sum_{j=0}^k {k \choose j}_{q}\xi^{(k-j)} \otimes' \xi^{(j)}\right) = \sum_{k=0}^n {n \choose k}_{p^q} F^*(\xi^{(n-k)_{q',y'}}) \otimes' F^*(\xi^{(k)_{q',y'}})
\]
in $\mathrm P \otimes'_{A} \mathrm P = A \otimes_{R} A \otimes_{R} A$.
This follows from proposition 3.5 of \cite{LeStumQuiros18} applied both to $\sigma$ and $\sigma^p$ since $F^*$ is a ring homomorphism.
\end{proof}

\begin{lem} \label{compro}
Assume that $R$ is $q$-divisible of $q$-characteristic $p$. If $F^*(A') \subset A^{\partial_{\sigma}=0}$, then we have
\begin{enumerate}
\item $A \otimes_{A^{\partial_{\sigma}=0}} A \simeq \mathrm P^{(\infty)}_{(p-1)_{\sigma}} (= A[\xi]/\xi^{(p)})$ and
\item
$A \otimes_{A^{\partial_{\sigma}=0}} A$ is a direct factor in $A \otimes_{A'} A$.
\end{enumerate}

\end{lem}

\begin{proof}
First of all, the condition $F^*(A') \subset A^{\partial_{\sigma} = 0}$ implies that there exists a natural surjection
\[
A \otimes_{A'} A \twoheadrightarrow A \otimes_{A^{\partial_{\sigma}=0}} A.
\]
On the other hand, proposition \ref{epimono} provides a canonical injection $A \otimes_{A^{\partial_{\sigma}=0}} A \hookrightarrow \mathrm P^{(\infty)}_{(p-1)_{\sigma}}$.
Let us consider now the following commutative diagram
\[
\xymatrix{
A'[\xi] \ar[r] \ar[d]^{F^*} & A' \otimes_{R} A' \ar[d]^{F^*} \ar[r] & A' \ar[d] \\
A[\xi] \ar[r] & A \otimes_{R} A \ar[r] & A \otimes_{A'} A.
}
\]
The upper line sends $\xi$ to $0$ and it follows that the bottom line sends $\xi^{(p)} := F^*(\xi)$ to $0$.
In the end, we obtain the commutative diagram
\begin{equation} \label{factf}
\xymatrix{A \otimes_{A'} A \ar@{->>}[r] & A \otimes_{A^{\partial_{\sigma}=0}} A \ar@{^{(}->}[d] \\
A[\xi]/\xi^{(p)} \ar[r]^{\simeq} \ar[u] & \mathrm P^{(\infty)}_{(p-1)_{\sigma}}}
\end{equation}
from which both assertions follow.
\end{proof}

\begin{dfn}
We say that $F^*$ is \emph{adapted to $\sigma$} if $F^*$ finite flat or rank $p$ and $F^*(A') \subset A^{\partial_{\sigma}=0}$.
\end{dfn}

\begin{xmp}
If $R$ has $q$-characteristic $p$ and $A$ is a localization of the polynomial ring $R[x]$, then $F^*$ is always adapted.
\end{xmp}

\begin{prop} Assume that $R$ is $q$-divisible of $q$-characteristic $p$.
If $F^*$ is adapted to $\sigma$, we can make the identifications
\[
A \otimes_{A'} A = A \otimes_{A^{\partial_{\sigma}=0}} A = \mathrm P^{(\infty)}_{(p-1)_{\sigma}} = \mathrm P^{(0)}_{(p-1)_{\sigma}} = A[\xi]/\xi^{(p)}.
\]
\end{prop}

\begin{proof}
Only the first equality needs a proof.
We know from the second part of lemma \ref{compro} that $A \otimes_{A^{\partial_{\sigma}=0}} A$ is a direct factor in $A \otimes_{A'} A$.
But the first part of the lemma tells us that $A \otimes_{A^{\partial_{\sigma}=0}} A$ is free of rank $p$ over $A'$ and our assumption implies that $A \otimes_{A'} A$ is locally free of the same rank $p$ over $A'$.
Therefore, they must be equal.
\end{proof}

\begin{rmk}
By duality, lemma \ref{compro} tells us that, when $F^*(A') \subset A^{\partial_{\sigma}=0}$, we have $\overline {\mathrm D}_{A/R,\sigma} = \mathrm{End}_{A^{\partial_{\sigma}=0}} (A)$ and that this is a direct factor in $\mathrm{End}_{A'}(A)$.
Moreover, the proposition says that when $F^*$ is adapted to $\sigma$, then all three rings are equal.
As a consequence, we will actually have an equality $F^*(A') = A^{\partial_{\sigma}=0}$.
\end{rmk}

We denote by $\widehat {\mathrm Z}_{A/R,\sigma}$, $\widehat{\mathrm ZA}_{A/R,\sigma}$ and $\widehat{\mathrm D}_{A/R,\sigma}$ the completions with respect to $\partial_{\sigma}^p$ (or $\partial_{\sigma}$ for the last one: it gives the same thing).
We may now state our \emph{Azumaya splitting} result:

\begin{thm} Assume that $R$ is $q$-divisible of positive $q$-characteristic $p$.
If $F^*(A') \subset A^{\partial_{\sigma} = 0}$, then $\Phi_{A/R}$ provides an $A$-linear ring homomorphism
\begin{equation} \label{azsplit1}
\mathrm D_{A/R,\sigma} \to \mathrm{End}_{\mathrm Z_{A/R,\sigma}}(\mathrm ZA_{A/R,\sigma}).
\end{equation}
If moreover,  $F^*$ is finite flat of rank $p$, we obtain an isomorphism
\begin{equation} \label{azsplit}
\widehat {\mathrm D}_{A/R,\sigma} \simeq \mathrm{End}_{\widehat {\mathrm Z}_{A/R,\sigma}}(\widehat{\mathrm ZA}_{A/R,\sigma}).
\end{equation}
\end{thm}

Recall that the conjunction of the hypothesis exactly means that $F^*$ is adapted to $\sigma$.

\begin{proof}
Using lemma \ref{compro}, we deduce from proposition \ref{frobis} that there exists a canonical morphism of $A$-algebras
\[
(A \otimes_{A'} A) \otimes_{A'} A\langle \omega \rangle \to A \langle \xi \rangle.
\]
which is an isomorphism when $F^*$ is finite flat of rank $p$.
Moreover, it follows from proposition \ref{comdlef} that this morphism is compatible with the comultiplication maps.
Therefore, it produces by duality (and base change) an $A$-linear homomorphism of rings:
\[
\widehat {\mathrm D}_{A/R,\sigma} \to \mathrm{End}_{A'}(A) \otimes_{A'} \widehat{\mathrm ZA}_{A/R,\sigma} \to \mathrm{End}_{\widehat {\mathrm Z}_{A/R,\sigma}}(\widehat{\mathrm ZA}_{A/R,\sigma}).
\]
Actually, since the twisted divided $p$-Frobenius is continuous, this map is defined before completion.
Finally, when $F$ is finite flat, the last map is also an isomorphism and we are done.
\end{proof}

\begin{rmk}
\begin{enumerate}
\item
The first assertion of the theorem means that $\Phi_{A/R}$ turns $\mathrm ZA_{A/R,\sigma}$ into a $\mathrm D_{A/R,\sigma}$-module via
\[
\partial_{\sigma}^k \cdot z\partial_{\sigma}^{pi} = \Phi(\partial_{\sigma}^k \circ z)\partial_{\sigma}^{pi}.
\]
\item
As a consequence of the theorem, we see that when $F^*(A') \subset A^{\partial_{\sigma} = 0}$, the map $\Phi_{A/R}$ induces an endomorphism of the $A$-algebra $\mathrm Z_{A/R,\sigma}$, and that it gives rise to an automorphism of $\widehat {\mathrm Z}_{A/R,\sigma}$ if moreover $F$ is finite flat of rank $p$.
\item When $F$ is adapted to $\sigma$, we actually have an isomorphism (before completion)
\[
\mathrm Z_{A/R,\sigma}\ {}_{{}_{\Phi}\nwarrow}\!\! \otimes_{\mathrm Z_{A/R,\sigma}}\mathrm D_{A/R,\sigma} \simeq \mathrm{End}_{\mathrm Z_{A/R,\sigma}}(\mathrm ZA_{A/R,\sigma}).
\]
\end{enumerate}
\end{rmk}

\begin{xmp}
Since $\partial_{\sigma} \circ z = \partial_{\sigma}(z) + \sigma(z)\partial_{\sigma}$ for all $z \in A$, we have for all $n \in \mathbb N$, $\partial_{\sigma} \circ x^n = (n)_{q} x^{n-q}+ q^nx^n\partial_{\sigma}$.
It follows that
\[
\partial_{\sigma} \cdot 1 = \Phi(\partial_{\sigma}) = x^{p-1}\partial_{\sigma}^p,
\]
and for $n \geq 1$,
\[
\partial_{\sigma} \cdot x^n = \Phi(\partial_{\sigma}\circ x^k) = (n)_{q} x^{n-1} + q^nx^nx^{p-1}\partial_{\sigma}^p = \left((n)_{q} + q^nx^{p}\partial_{\sigma}^p\right)x^{n-1}.
\]
In other words, the matrix of $\partial_{\sigma}$ will be
\[
\left[\begin{array}{ccccc}0 & qx^{p}\partial_{\sigma}^p + 1 & 0 & \cdots & 0 \\\vdots & 0 & \ddots & \ddots & \vdots \\\vdots & \vdots & \ddots & \ddots & 0 \\0 & 0 & \cdots & 0 & q^{p-1}x^{p}\partial_{\sigma}^p + (p-1)_{q}   \\ \partial_{\sigma}^p & 0 & \cdots & \cdots & 0\end{array}\right].
\]
Note that this is slightly different from the formulas of proposition 4.1 of \cite{GrosLeStum13} because we use here Ogus-Vologodsky divided Frobenius (the coefficients $B$ and not the coefficients $C$).
\end{xmp}

\begin{cor} Assume that $R$ is $q$-divisible of positive $q$-characteristic $p$ and that $F$ is adapted to $\sigma$.
Then, $\Phi_{A/R}$ induces an equivalence between $\widehat {\mathrm D}_{A/R,\sigma}$-modules and $\widehat {\mathrm Z}_{A/R,\sigma}$-modules.
\end{cor}

\begin{proof}
This is Morita equivalence. $ \Box$\end{proof}

In order to state the \emph{twisted Simpson correspondence}, we need to recall some vocabulary.
An endomorphism $u_{M}$ of an abelian group $M$ is said to be \emph{quasi-nilpotent} if
\[
\forall s \in M, \exists N \in \mathbb N, u_{M}^N(s) =0.
\]
Also, a \emph{$\sigma$-derivation} on an $A$-module $M$ is an $R$-linear map $\partial_{\sigma,M} : M \to M$ that satisfies the \emph{twisted Leibniz rule}
\[
\forall z \in A, \forall s \in M, \partial_{\sigma,M}(zs) = \partial_{\sigma_{A}}(z)s + \sigma(z) \partial_{\sigma,M}(s).
\]

We can now reformulate the previous corollary in more down-to-earth terms:

\begin{cor} [Twisted Simpson correspondence] Assume that $R$ is $q$-divisible of positive $q$-characteristic $p$ and $F$ is adapted to $\sigma$.
Then, the category of $A$-modules $M$ endowed with a quasi-nilpotent $\sigma$-derivation $\partial_{\sigma,M}$ is equivalent to the category of $A'$-modules $H$ endowed with a quasi-nilpotent $A$-linear endomorphism $u_{H}$. 
\end{cor}

\begin{rmk}
\begin{enumerate}
\item
The equivalence is explicit and given by
\[
M \mapsto H := M^{\Phi = 1} \quad \mathrm{and} \quad H \mapsto M := A \otimes_{A'} H.
\]
More precisely,
\[
M^{\Phi = 1} := \{s \in M, \forall k \in \mathbb N, \Phi(\partial_{\sigma}^k)(s) = \partial_{\sigma}^k(s) \}
\]
(which is not easy to compute) will be endowed with the action of $\partial_{\sigma}^p$ and $A \otimes_{A'} H$ will be endowed with the unique $\sigma$-derivation such that
\[
\partial_{\sigma}(1 \otimes s) = x^{p-1} \otimes \theta(s).
\]
\item
Twisted Simpson correspondence holds for example in the following situations:
\begin{enumerate}
\item
$R$ a ring of prime characteristic $p$ and $x$ is an \'etale coordinate on $A$ (Ogus-Vologodsky).
\item
$R$ contains a field $K$, $q \in K$ is a primitive $p$th root of unity and $A = R[x]$ or $R[x, x^{-1}]$.
\item
$R$ is $p^N$-torsion with $p$ prime, the $p$th power map of $R/p$ lifts to $R$, $q$ is a non trivial $p$th root of unity and $x$ is an \'etale coordinate on $A$.
\end{enumerate}
\end{enumerate}
\end{rmk}

\addcontentsline{toc}{section}{References}
\printbibliography

\end{document}